\documentclass[12pt]{article}

\usepackage{style}
\usepackage{macros}
\usepackage{amsmath}
\usepackage{subfiles}

\usepackage[a4paper,margin=1.2in]{geometry}
\usepackage{fancyhdr}
\title{Two-dimensional Virasoro algebras}
\author{Zhengping Gui and Brian R. Williams}
\date{\today}

\newcommand{\SP}{\mathbb{P}^{S^1}}
\newcommand{\1}{{\textbf{\texttt{\tiny{1}}}}}
\newcommand{\2}{\textbf{\texttt{\tiny{2}}}}
\newcommand{\jj}{\textbf{\texttt{\tiny{j}}}}

\newcommand{\ii}{\textbf{\texttt{\tiny{i}}}}
\newcommand{\kkk}{\textbf{\texttt{\tiny{k}}}}
\newcommand{\rrr}{\textbf{\texttt{\tiny{r}}}}
\newcommand{\ttt}{\textbf{\texttt{\tiny{t}}}}
\newcommand{\s}{\textbf{\texttt{\tiny{s}}}}
\newcommand{\ddd}{\textbf{\texttt{\tiny{d}}}}
\newcommand{\smallhollowcolon}{%
  \vcenter{\offinterlineskip
    \ialign{##\cr
      \tiny$\circ$\cr
      \noalign{\kern0.4ex}
      \tiny$\circ$\cr
    }%
  }%
}

\newcommand{\dd}{\textbf{\texttt{\tiny{d}}}}

\newcommand{\tr}{\mathrm{tr}}

\newcommand{\permute}{\tau}

\usepackage{simpler-wick}
\usepackage{simpler-wick}
\usetikzlibrary{decorations.markings}
\tikzset{W->-/.style={decoration={
  markings,
  mark=at position 0.5*\pgfdecoratedpathlength+2pt with
  {\draw[-latex] (-2pt,0pt) -- (1pt,0pt);}},postaction={decorate}},
  W-<-/.style={decoration={
  markings,
  mark=at position 0.5*\pgfdecoratedpathlength with
  {\draw[latex-] (-2pt,0pt) -- (1pt,0pt);}},postaction={decorate}}
  }
\pgfkeys{
  /simplerwick/.cd,
  arrows/.store in=\LstWickArrows,
  arrows/.initial={-,-,-,-,-,-,-,-,-}, 
  arrows={-,-,-,-,-,-,-,-,-},
  positions/.store in=\LstWickPositions,
  positions={+1,+1,+1,+1,+1,+1,+1,+1,+1},
  positions/.initial={+1,+1,+1,+1,+1,+1,+1,+1,+1},
}

\makeatletter
\newcounter{Wick@up}
\newcounter{Wick@down}
\def\swick@end#1#2{
  \swick@setfalse@#1
  \tikzexternaldisable
  \begin{tikzpicture}[remember picture, baseline=(swick-close#1.base)]
    \node[use as bounding box, inner sep=0pt, outer sep=0pt] (swick-close#1) {$\displaystyle #2$};
  \end{tikzpicture}
  \tikz[remember picture, overlay]
{
\xdef\myW@style{\empty}
\foreach \W@X[count=\W@C] in \LstWickArrows
{\ifnum\W@C=#1
\xdef\myW@style{\W@X}
\fi}
\ifx\myW@style\empty
\PackageWarning{simpler-wick}{%
The list arrows has not enough entries!%
}{}
\xdef\myW@style{-}
\fi
\xdef\myW@pos{-77}
\foreach \W@X[count=\W@C] in \LstWickPositions
{\ifnum\W@C=#1
\xdef\myW@pos{\W@X}
\fi}
\ifnum\myW@pos=-77
\PackageWarning{simpler-wick}{%
The list positions has not enough entries!%
}{}
\xdef\myW@pos{+1}
\fi
\ifnum\myW@pos=-1
    \draw[\myW@style] ($(swick-open#1.south) + (0, -3pt)$) 
          -- ($(swick-open#1.base) + (0, -\swick@offset) + \theWick@down*(0, -\swick@sep)$) 
          -- ($(swick-close#1.base) + (0, -\swick@offset) + \theWick@down*(0, -\swick@sep)$) 
          -- ($(swick-close#1.south) + (0, -3pt)$);
\stepcounter{Wick@down}
\else
\stepcounter{Wick@up}
    \draw[\myW@style] ($(swick-open#1.north) + (0, 3pt)$) 
          -- ($(swick-open#1.base) + (0, \swick@offset) + \theWick@up*(0, \swick@sep)$) 
          -- ($(swick-close#1.base) + (0, \swick@offset) + \theWick@up*(0, \swick@sep)$) 
          -- ($(swick-close#1.north) + (0, 3pt)$);
\fi}
  \tikzexternalenable}
\def\wick@[#1]#2{\setcounter{Wick@up}{0}
\setcounter{Wick@down}{-1}
  \ifmmode
    \begingroup
    \pgfkeys{
        simplerwick,
        #1}
    \swick@cond@reset
    \swick@count=0
    \def\swick@max{0}
    \def\c{\swick@smart}
    #2
    \dimen0=\swick@sep
    \multiply\dimen0 by \swick@max
    \advance\dimen0 by \swick@offset
    \vbox to \dimen0{}
    \swick@cond@any{
      \PackageWarning{simpler-wick}{%
        I have reached the end of \protect\wick\space with some unclosed
        contractions%
      }{}
    }{}
    \endgroup
  \else
    \PackageWarning{simpler-wick}{%
      \protect\wich\space has been called outside a math environment, this will
      be ignore%
    }
  \fi
}

\makeatother

\begin{document}
\maketitle

\abstract{We classify central extensions of the dg Lie algebra of derived global sections of the tangent sheaf on the
punctured, formal $2$-disk.
We then prove a local, universal form of the Grothendieck--Riemann--Roch theorem for families of two-dimensional complex
varieties.}


\section*{Introduction}

The Virasoro algebra is ubiquitous in conformal field theory, and more generally representation theory as a whole.
Geometrically, the starting point is the \textit{Witt} Lie algebra: the Lie algebra of vector fields on the formal punctured disk.
The Virasoro algebra is the unique central extension of the Witt algebra.
In other words, representations of the Virasoro algebra are in bijective correspondence with projective representations
of the Witt algebra.
The action by this preferred central element $c \in \C$ is called the \textit{central charge} of the representation.
It is a powerful invariant in conformal field theory.

Our starting point is the punctured $d$-disk $\mathring{D}^d$ in dimension $d>1$.
The key feature is that $\mathring{D}^d$ is not an affine for $d>1$, so, in particular the Lie algebra of vector fields on
$\mathring{D}^d$ is equal to the Lie algebra of vector fields on the (unpunctured) $d$-disk:
\begin{equation}\label{}
  \lie{w}_d \cong \Gamma(\mathring{D}^d, \sT) .
\end{equation}
This Lie algebra certainly does not play the same role in dimension $d>1$ as the Witt algebra does in
dimension $d=1$.
For example, it is well-known that $H^2(\lie{w}_d) = 0$, for $d>1$, so this Lie algebra admits no nontrivial central extensions.
Following \cite{FHK,HKgf} we instead consider the \textit{derived} global sections of the tangent sheaf 
\begin{equation}\label{}
  \lie{witt}_d \simeq \R \Gamma(\mathring{D}^d,\sT) .
\end{equation}
We write equivalence here since in the main text, see section \ref{s:witt}, we construct an explicit dg Lie model for the higher-dimensional Witt algebra
using the ``Jouanolou method'' as in \cite{BD,FHK,GWWchiral}.

In \cite{FHK} the higher Kac--Moody algebras were introduced as central extensions of the derived global
sections of $\lie{g}$ multivariate currents $\R \Gamma(\mathring{D}^d, \lie{g}
\otimes \cO)$ where $\lie{g}$ is a Lie algebra.
Further, it was shown in \cite{GWkm} how such extensions appear as quantum \textit{symmetries} in higher-dimensional
holomorphic field theories which have $\lie{g}$ as a classical symmetry.
Likewise, central extensions of $\lie{witt}_d$ have been conjectured to appear as fundamental quantum symmetries in
higher-dimensional holomorphic field theory \cite{BWac,KapranovNotes}.
We call such central extensions \textit{higher-dimensional Virasoro algebras}.

More geometrically, the higher Virasoro algebras should be involved with a sort of \textit{uniformization} for the
moduli of complex structures in higher-dimensions.
This was worked out for the higher-dimensional Kac--Moody algebras and the moduli of $G$-bundles in \cite{FHK}.
In another direction, we have also shown how to realize such higher Virasoro algebras using higher-dimensional chiral
algebras \cite{GWWchiral}.

In this paper we take a (derived) algebraic approach to higher-dimensional Virasoro algebras.
Many of our constructions work for general dimension $d$, including our definition of the higher-dimensional Virasoro
algebras.
However, our main classification theorem applies only to the case $d=2$.
We treat the $d > 2$ cases of our main theorems in a future publication.

Our first result concerns the classification of central extensions, and hence the second Lie algebra (hyper)cohomology, of $\lie{witt}_d$.
As conjectured by Kapranov in \cite{KapranovNotes}, see also \cite{Wthesis,BWac}, it is expected that central extensions of the dg Lie algebra
$\lie{witt}_d$ are in bijective correspondence with degree $2d+2$ polynomials:
\begin{equation}\label{}
  \C[\op{ch}_1,\op{ch}_2,\ldots,\op{ch}_d]_{2d+2}
\end{equation}
where $c_i$ carries degree $2i$.
These classes have also appeared in \cite{BWvf,BWac}.
Our first main result produces the candidate classes explicitly, which holds for general dimension $d$.
When $d=2$ we show that these are, in fact, all of the classes up to equivalence.

\begin{globaltheorem}\label{thm:chern}
For any integer $d \geq 1$, there is an explicit homomorphism 
\begin{equation}\label{}
  \C[\op{ch}_1,\ldots,\op{ch}_d]_{2d+2} \to \HH^2_{Lie}(\lie{witt}_d)
\end{equation}
which is an isomorphism when $d=1,2$.
\end{globaltheorem}

Of course, the $d=1$ result is classical.
The $d=2$ result states that space of central extensions for $\lie{witt}_2$, up to
equivalence, is two-dimensional corresponding to basis vectors $\op{ch}_1^3, \op{ch}_1 \op{ch}_2$.
In general, the formulas we give for the cocycles which represent these universal characteristic classes are given as
explicit residues over the higher-dimensional sphere in punctured affine space. 
See section \ref{s:chern} for details.

From the construction of the map in theorem \ref{thm:chern}, we arrive at the definition of the higher-dimensional Virasoro algebra in any dimension,
though our result strictly only classifies them when $d=2$.
We expect that this map is an isomorphism for all $d$.

Our next main result is a local and universal version of the Grothendieck--Riemann--Roch theorem.
We briefly formulate this in dimension $d=1$, where this formula is well-known.
Suppose that $\lambda$ labels a $\lie{gl}(1)$ weight and let $\cK^\lambda$ be the corresponding coinduced $\op{witt}
_1$-module, denoted $\sfL_\lambda \colon \op{witt}_1 \to \op{End}(\cK^\lambda)$.
As $\cK^\lambda$ is a Tate vector space (as a vector space it is isomorphic to Laurent series $\C(\!(z)\!)$, the second
Lie algebra cohomology of $\op{End}(\cK^\lambda)$ is one-dimensional \cite{FHK}, with a distinguished generator that we
denote $\mathfrak{t}$ that is linearly dual to the element $\frac{1}{z} \wedge z$.
The local Grothendieck--Riemann--Roch theorem is a formula for the restriction of this universal class along the
representation $\sfL_\lambda$:
\begin{equation}\label{}
  \sfL_\lambda^* \mathfrak{t} = \frac12 (6 \lambda^2 - 6 \lambda + 1) c_1^2 \in \HH^2(\lie{witt}_1).
\end{equation}
This universal class $c_1^2$ reflects the formal version of the
Miller-Mumford-Morita class~$\kappa \in H^2 \left(\br \cM_g\right)$.

We replace $\lie{witt}_1$ with $\lie{witt}_2$ and $\cK^\lambda$ with the Jouanolou model of a tensor $\lie{gl}_2
$-module.
For simplicity, we take the trivial module as input, and so consider the natural $\lie{witt}_2$-module structure on the
Jouanolou model for the structure sheaf
\begin{equation}\label{}
  \cJ_2 \simeq \R \Gamma(\mathring{D}^2 , \cO) 
\end{equation}
  by derivations. (We point out the generalization to general dimension $d$ and general tensor modules in section \ref{s:2dGRR}, see
  also \cite{KapranovNotes}).
  Denote this module structure by $\sfL \colon \lie{witt}_2 \to \op{End}(\cJ_2)$.
  A theorem of \cite{FHK} states that $\cJ_2$ is a \textit{dg Tate} vector space and it follows that the
  second degree (hyper) Lie algebra cohomology of its endomorphisms is one-dimensional: 
  \begin{equation}\label{}
    \HH^2_{Lie}( \op{End} \cJ_2) \simeq \C .
  \end{equation}
  with a distinguished generator $\mathfrak{t}$ dual to $z^{\texttt{1}} \wedge z^{\texttt{2}} \wedge P$.
  Where $P$ denotes multiplication by the Jouanolou Bochner--Martinelli form which satisfies $\op{Res}(P \wedge \d z^1
  \wedge \d z^2) = 1$.
  \begin{globaltheorem}\label{thm:GRRglobal}
  One has 
  \begin{equation}
    \sfL^*\mathfrak{t} = \op{Td}|_{6}=\frac{1}{48} \left(\op{ch}_1^3 - 2 \op{ch}_1 \op{ch}_2\right)\in \HH^2_{Lie}(\lie{witt}_2) .
  \end{equation}
\end{globaltheorem}

We prove this theorem in sections \ref{s:2dgf} and \ref{s:2dGRR}.

\subsection*{Acknowledgements}
We thank Kevin Costello, Mikhail Kapranov, and Matt Szczesny for comments on the first draft of this paper.

\section{dg Witt algebra}\label{s:witt}

The Lie algebra of vector fields $\op{Der}(\C[z^\1,\ldots,z^{\ddd}])$ on $\AA^d$ is freely generated over $\C[z^\1,
\ldots,z^\ddd]$ by coordinate derivations
$\del_\1,\ldots,\del_\ddd$.
The variety $\AA^d - \{0\}$, on the other hand, is not affine and therefore taking global sections is not an exact
functor.
For example, the global sections of the tangent sheaf on $\AA^d - \{0\}$ is identical to the Lie algebra of vector
fields on $\AA^d$.
The cohomology of $\AA^d - \{0\}$ with coefficients in the tangent sheaf is concentrated in degrees zero and $d-1$.

In this section, we introduce an explicit dg model for the derived global sections of the tangent sheaf on punctured
affine space which we will refer to as the dg Witt Lie algebra.
When $d=1$ it recovers the ordinary Witt algebra $\op{Der} \C(\!(z)\!)$, but when $d>1$ it is necessarily a derived object.
While we will primarily focus on the case $d=2$ in subsequent sections, we will allow for general $d$ for most of this
section.

\subsection{Jouanolou models}

Refer to \cite[\S 1.2]{FHK} and \cite[\S 1]{GWWchiral} for notations.
Denote $\bJ_{\mathring{\AA}^d}(\sM)$ the \textit{Jouanolou model} of the sheaf $\sM$ on punctured $d$-dimensional affine space 
is the Jouanolou torsor for $\AA^d - \{0\}$ 
The complex $\bJ_{\mathring{\AA}^d}(\sM)$ is an explicit model for the derived global sections $\R \Gamma(\mathring{\AA}^d,
\sM)$ \cite{BD,FHK}.
We denote its internal differential by $\dbar$ (usually with the sheaf $\sM$ understood from context).

When $\sM = \sO$ denote this complex by $\cJ^{poly}_d$ (in \cite{FHK} this is denoted $A_{[d]}$).
As the Jouanolou construction is monoidal, $\cJ^{poly}_d$ is naturally a commutative dg algebra which is concentrated in degrees $0,\ldots,d-1$.

For any quasi-coherent sheaf $\sM$, the complex $\bJ_{\mathring{\AA}^d}(\sM)$ is a dg $\cJ^{poly}_d$-module.
When $\sM = \sT$, the tangent sheaf, we will denote 
\begin{equation}\label{}
  \lie{witt}^{poly}_d \define \bJ_{\mathring{\AA}^d}(\sT) .
\end{equation}
We will also denote the internal differential for the dg Lie algebra $\lie{witt}^{poly}_d$ by $\dbar$.

\begin{lemma}\label{lem:triv}
  As a dg $\cJ^{poly}_d$-module $\lie{witt}^{poly}_d$ is freely generated by $d$-generators $\{\del_\1,\ldots,\del_\ddd\}$.
\end{lemma}
\begin{proof}
  Denote by $\del_\jj = \frac{\del}{\del z^\jj}$, $\jj=1,\ldots,d$ a basis for the tangent space.
  It follows from the same logic as \cite[\S 1.2]{FHK}, a degree $q$ element of $\lie{witt}^{poly}_d$ is of the form 
  \begin{equation}\label{eq:expansion}
   \sum_i \sum_{|I|=q} \alpha_{I}^\ii(z,z^*) \d z^{*I} \del_\ii
  \end{equation}
  where $\alpha_I^\ii \in \C[z,z^*][(zz^*)^{-1}]$, $\d z^{*I} = \d z^{*\ii_1} \cdots \d z^{*\ii_k}$, and the sum is over $\ii=1,\ldots,
  d$ and multi-indices $I = (\ii_1 \leq \ii_1 < \ldots < \ii_d\leq d$). 
\end{proof}


We will mostly be considering objects on the \textit{formal} $d$-disk.
Let $\lie{witt}_d$ denote the formal completion of $\lie{witt}_d^{poly}$ obtained by letting the coefficients
  $\alpha_I^\ii(z,z^*)$ in the expression \eqref{eq:expansion} be valued in 
  \begin{equation}\label{}
    \alpha_I^\ii \in \C\llbracket z\rrbracket[z^*][(zz^*)^{-1}]
  \end{equation}
Similarly, for the Jouanolou model of the structure sheaf, we let $\cJ_d$ be the formal completion of $\sJ^{poly}_d$
($\cJ_d$ is denoted $A_d$ in \cite{FHK}).
The cochain complex $\lie{witt}_d$ is a $\cJ_d$-module in the obvious way.
We will give an even more explicit description of $\lie{witt}_d$ in the case $d=2$ at the end of this section.

\subsection{Graded Jacobi bracket}
Abstractly, any model for the derived global sections of the tangent sheaf can be endowed with the structure of a homotopy Lie algebra.
For $\lie{witt}_d$, the following explicit bracket endows it with the structure of a dg Lie algebra.

\begin{lemma}
  Together with the internal Jouanolou differential $\dbar$, the formula
\begin{multline}
  \left[\alpha_{I,\ii} (z,z^*) \d x^I \del_\ii, \beta_{J,\jj} (z,z^*) \d z^{*J}  \del_{\jj} \right]
  \\ = \left(\alpha_{I,\ii}(z,z^*) \frac{\del \beta_{J,\jj}}{\del z^\ii}(z,z^*)\; \del_\jj - (-1)^{|I|} \beta_{J,\jj}(z,
z^*) \frac{\del \alpha}{\del z^\jj} (z,z^*) \; \del_\ii\right) \d z^{*I} \d z^{*J} 
\end{multline}
endows $\lie{witt}_d$ with the structure of a dg Lie algebra.
It is compatible with the Jacobi bracket of holomorphic sections of the tangent sheaf~$\sT$.
\end{lemma}

\begin{definition}[]
\label{dfn:witt}
With this bracket and differential we refer to the dg Lie algebra $\lie{witt}_d$ as the $d$-dimensional \defterm{Witt
dg Lie algebra}. 
\end{definition}

On the other hand, letting the symbols~$\del_\jj$ act by vector fields, we see that $\lie{witt}_d$ acts on $\cJ_d$ by
derivations.
This leads to the following, purely algebraic, interpretation of the dg Witt algebra.

\begin{proposition}
\label{prop:derivations}
 There is a quasi-isomorphism of dg Lie algebras 
 \begin{equation}\label{}
   \lie{witt}^{poly}_d \simeq \op{Der}(\cJ^{poly}_d)
 \end{equation}
 and similarly $\lie{witt}_d \simeq \op{Der}(\cJ_d)$. 
\end{proposition}
\begin{proof}
  From the explicit presentation of $\cJ^{poly}_d$ in \cite{FHK} we see that the $\sJ_d^{poly}$-module $\op{Der}
  (\cJ^{poly}_d)$ is generated by degree zero derivations $\del_\1,\ldots,\del_\ddd, \del_{z^{*\1}}, \ldots, \del_{z^{*\ddd}}$ and
  degree $-1$ derivations
  derivations $\del_{\d z^{*\1}}, \ldots, \del_{\d z^{*\ddd}}$.
  Therefore, there is a natural map 
  \begin{equation}\label{}
    \lie{witt}_d^{poly} \to \op{Der}(\cJ_d^{poly}).
  \end{equation}
  To see that this is a quasi-isomorphism we use Cartan's magic formula 
  \begin{equation}\label{}
    [\dbar , \del_{\d z^{*\jj}}] = \del_{z^{*\jj}} .
  \end{equation}
  The second claim follows from taking completions.
\end{proof}

Similarly, for any sheaf of $\sT$-modules $\sM$, the Jouanolou model $\bJ_{\mathring{\AA}^d}(\sM)$ is a dg $\lie{witt}
_d$-module.
(These are the ``tensor" modules.)

\subsection{Gelfand--Fuks cohomology}

For a dg Lie algebra $\lie{g}$ we denote $C^\bu(\lie{g})$ its Chevalley--Eilenberg complex computing Lie algebra
cohomology with trivial coefficients.
If $M$ is a dg $\lie{g}$-module then let $C^\bu(\lie{g};M)$ be cohomology of $\lie{g}$ with coefficients in $M$.
This latter complex has the differential of the form 
\begin{equation}\label{}
  \left(\d_{\lie{g}} + \d_M\right) +  \d^{\rm Lie}
\end{equation}
where $\d_{\lie{g}}$ is the linear dual of the internal differential of the dg Lie algebra $\lie{g}$, $\d_M$ is the
internal differential for the dg module $M$, and $\d^{\rm Lie}$ is
Chevalley and Eilenberg's differential defined on generators by
\begin{itemize}
  \item For $v \in M$, $\d^{\rm Lie} v (x) = x \cdot v$.
  \item For $\varphi \colon \lie{g} \to M$, $\d^{\rm Lie}\varphi (x,y) = \varphi([x,y]) + x \cdot \varphi(y) - (-1)
    ^{|x||y|} y \cdot \varphi(x)$.
  \end{itemize}
We will denote the total cohomology of this complex by $\HH_{\rm Lie}^\bu(\lie{g} ; M)$.
Consider the increasing fitration by graded symmetric polynomial degree on $C^\bu(\lie{g};M)$.
This defines a spectral sequence which converges to the total cohomology.
Observe that the first two differentials $(\d_{\lie{g}}+\d_M)$ preserve the grading by graded symmetric degree, thus
the $E_1$ page of this spectral sequence is simply
\begin{equation}\label{}
  H^\bu \left( \op{Hom} (\op{S}(\lie{g}[1], M) \;\; , \;\; \d_{\lie{g}} + \d_{M} )\right) .
\end{equation}

Central extensions (extensions by $\C$) of the dg Lie algebra are parametrized by the total degree two cohomology
\begin{equation}\label{}
  \HH_{\rm Lie}^2(\lie{g}) .
\end{equation}
Since $\lie{g}$ is equipped with its own internal grading, total degree two classes may be represented by non-quadratic
cocycles in which case we present the central extension most efficiently as an $L_\infty$-algebra.
In the next section we will characterize central extensions of the dg Witt algebra by computing
\begin{equation}\label{}
  \HH^2_{\rm Lie}(\lie{witt}_d) .
\end{equation}
In dimension $d$ we find nontrivial cocycle representatives which are homogenous $(d+1)$-linear.

Gelfand and Fuks study the continuous Lie algebra cohomology of the Lie algebra of formal vector fields on the $d$-disk \cite{GF,
FuksBook}.
We will denote this lie algebra by~$\lie{w}_d$.
Notice that it is naturally a Lie subalgebra of the dg Lie algebra $\lie{witt}_d$.

\subsection{Explicit presentation $d=2$}

We unpack the presentation of $\lie{witt}_2$.
First, we recall the description of $\cJ_2$ following \cite[\S 1]{GWWchiral}.
See \cite{FHK} for original definitions.

The degree zero part of the dg algebra $\cJ_2$ is 
\[
  \cJ_2^0 = \C\llbracket z^\1, z^\2\rrbracket[x^\1,x^\2] \slash (z^\1 x^\1 + z^\2 x^\2 = 1) 
\]
which involves two formal variables and two polynomial variables subject to the relation. By substituting $x^{\ii}=\frac{z^{*\ii}}{zz^*}=\frac{z^{*\ii}}{z^{\1}z^{*\1}+z^{\2}z^{*\2}}$, we get the original definition in \cite{FHK}. The degree one part $\cJ_2^1$ is generated over $\cJ_2^0$ by a single element that we denote $P \in \cJ_2^1$.
We have $\cJ_2^k = 0$ for $k > 2$.
The differential is prescribed on generators by 
\begin{align*}
  \dbar (z^\1) &= \dbar (z^\2) = 0 \\ 
  \dbar (x^\1) &= - z^\2 P, \quad \dbar (x^\2) = z^\1 P .
\end{align*}
In other words,
\[
  \cJ_2 = \C\llbracket z^\1, z^\2\rrbracket[x^\1,x^\2, \dbar (x^\1), \dbar (x^\2)] \slash (z^\1 x^\1 + z^\2 x^\2 = 1,z^\1 \dbar (x^\1) + z^\2 \dbar (x^\2) = 0),
\]
and the element $P$ corresponds to the Martinelli-Bochner form
\[
P=x^\1 \dbar (x^\2)-x^\2 \dbar (x^\1) =\frac{z^{*\1}dz^{*\2}-z^{*\2}dz^{*\1}}{(z^{\1}z^{*\1}+z^{\2}z^{*\2})}.
\]

The model $\cJ_2$ carries a natural action of differential operators 
\[
\del_{\ii}x^{\jj}=\del_{\ii}\left(\frac{z^{*\jj}}{z^{\1}z^{*\1}+z^{\2}z^{*\2}}\right)=-x^{\ii}\cdot x^{\jj},\quad \ii,\jj=1,2.
\]
By direct computation, we have
\[
(\del_{\1})^{k}(\del_{\2})^{\ell}P=(-1)^{k+\ell}\cdot (k+\ell+1)!\cdot (x^{\1})^k(x^{\2})^{\ell}\cdot P.
\]

The $z$-variables comprise the zeroth cohomology $H^0_{\dbar}(\cJ_2) = \C\llbracket z^\1,z^\2\rrbracket$ whereas the first cohomology can be
identified with 
\begin{equation}\label{}
  H^1_{\dbar}(\cJ_2) \cong \C[\del_\1,\del_\2] P .
\end{equation}
where we introduce the cohomology classes 
\[
  \del_\1^{(k)} \del_\2^{(\ell)} P \define \frac{(k+\ell+1)!}{k!\ell!}\cdot \left[(x^\1)^k (x^\2)^\ell P \right]_{\dbar} .\footnote{This is a small simplification of
    notation since $\lie{w}_2$, and hence the algebra of differential operators on the $2$-disk, acts on $H^1_{\dbar}(\cJ_2)$.
    With respect to this action one has $\del_\1^{(k)} \del_\2^{(\ell)} P = \frac{(-1)^{k+\ell}}{k! \ell!} \frac{\del^{k+\ell}P}
  {\del (z^\1)^k \del (z^\2)^\ell}$.} 
\]
In other words our notation $\del_\1^k \del_\2^\ell P$ is for the class representing the cocycle $(x^\1)^k (x^\2)^\ell P \in Z(\cJ_2^1)$.
The following lemma is easy to prove, and is one of the main advantages of using our particular model for punctured
affine space.
(The analogous statement is true for general dimension $d$.)

\begin{proposition}
  The Jouanolou model $\cJ_2$ is $GL_2$-equivariant, where the element $P$ is $GL_2$-invariant.
  Let $V$ be the $2$-dimensional fundamental representation of $GL_2$.
  The corresponding decomposition of cohomology is: 
  \begin{equation}\label{}
    H^0_{\dbar}(\cJ_2) \cong \Hat{\op{S}} (V^*) 
  \end{equation}
  where the hat denotes completed symmetric algebra, and  
  \begin{equation}\label{}
    H^1_{\dbar}(\cJ_2) \cong \op{S}(V) \otimes \wedge^2 V .
  \end{equation}
\end{proposition}

Then, the residual graded commutative algebra structure on $H^\bu_{\dbar}(\cJ_2)$ is determined by the $\C\llbracket
z^1,z^2\rrbracket
$-module structure on $H^1_{\dbar}(\cJ_2)$.
It is given by 
\begin{equation}\label{}
  (z^\1)^{k_1} (z^\2)^{k_2} \cdot \left( \del_\1^{(\ell_1)} \del_\2^{(\ell_2)} P \right) =
  \begin{cases}
    \del_\1^{(\ell_1 - k_1)} \del_\2^{(\ell_2 - k_2)} P, & \ell_i \leq k_i \\
    0 & \text{else}
  \end{cases}
\end{equation}

We turn to the dg Witt Lie algebra $\lie{witt}_2$. 
As we showed, it is free over $\cJ_2$ on two generators 
\begin{equation}\label{}
  \lie{witt}_2 = \cJ_2 \otimes \C\{\del_\1,\del_\2\} .
\end{equation}
Thus, any element in $\lie{witt}_2$ has the form 
\begin{equation}\label{}
  T = \left(f^\1(z,x) + g^\1(z,x) P\right)\del_\1 + \left(f^\2(z,x) + g^\2(z,x)P\right) \del_\2 
\end{equation}
where $f^\1,f^\2,g^\1,g^\2 \in \cJ_2^0$.
Succinctly, we will write this vector field as $T = T^\ii \del_\ii$ where $T^\ii = f^\ii + g^\ii P$.
The differential $\dbar$ of $\lie{witt}_2$ applied to such an element is 
\begin{equation}\label{}
  \dbar T = (\dbar f^\1) \del_\1 + (\dbar f^\2) \del_\2 
\end{equation}
where $\dbar f^\1, \dbar f^\2$ are defined above.

As a corollary of the above $GL_2$-decomposition we obtain the similar result.\footnote{This proposition pertains to
  the internal, or $\dbar$-cohomology of $\lie{witt}_2$. And, notably, \textit{not} the Lie algebra cohomology of $\lie{witt}
_2$.}
\begin{corollary}
  The dg Lie algebra $\lie{witt}_2$ is $GL_2$-equivariant.
  The corresponding decomposition of cohomology is: 
  \begin{equation}\label{}
    H^0(\lie{witt}_2) \cong \Hat{\op{S}} (V^*) \otimes V 
  \end{equation}
  and 
  \begin{equation}\label{}
    H^1(\lie{witt}_2) \cong \op{S}(V) \otimes \wedge^2 V \otimes V
  \end{equation}
\end{corollary}

In fact, we have the following stronger result.
We introduce the Lie algebra $\lie{w}_2$ of \textit{formal} vector fields on the $2$-disk studied in \cite{FuksBook},
for example.

\begin{lemma}[]
\label{prop:dbarwitt}
  There is an isomorphism of graded Lie algebras (concentrated in degrees zero and one)
  \begin{equation}\label{}
    H^\bu(\lie{witt}_2) \cong \lie{w}_2 \ltimes \sP [-1] 
  \end{equation}
  where the $\lie{w}_2$-module structure on 
  \begin{equation}\label{}
    \sP \define H^1_{\dbar} (\cJ_2) \otimes \C\{\del_\1,\del_\2\}
  \end{equation}
  is defined by  
  \begin{align*}\label{}
    (z^\1)^{a+1} (z^\2)^b \del_\1 \cdot \left(\del_\1^{(k)} \del_\2^{(\ell)} P \del_\1\right) & = -(k+a+2) \del_\1^{(k-a)} \del_\2^{(\ell-b)} P
    \del_\1 \\ 
    (z^\1)^{a+1} (z^\2)^b \del_\1 \cdot \left(\del_\1^{(k)} \del_\2^{(\ell)} P \del_\2\right) & = -(k+1) \del_\1^{(k-a)} \del_\2^{(\ell-b)}
    P \del_\2 - b \del_\1^{(k-a-1)} \del_\2^{(k-b+1)} P \del_\1 \\
    (z^\1)^{a} (z^\2)^{b+1} \del_\2 \cdot \left(\del_\1^{(k)} \del_\2^{(\ell)} P \del_\1\right) & = -(\ell+1) \del_\1^{(k-a)} \del_\2^{(\ell-b)}
    P \del_\1 - a \del_\1^{(k-a+1)} \del_\2^{(k-b-1)} P \del_\2 \\
    (z^\1)^{a} (z^\2)^{b+1} \del_\2 \cdot \left(\del_\1^{(k)} \del_\2^{(\ell)} P \del_\2\right) & = -(\ell+b+2) \del_\1^{(k-a)} \del_\2^{(\ell-b)} P
    \del_\2 . 
  \end{align*}
  In these expressions, $a,b,k,\ell \geq 0$ and we declare $\del_\1^{(m)} \del_\2^{(n)} P = 0$ if $m < 0$ or $n < 0$.
\end{lemma}

The continuous linear dual of the module 
\begin{equation}\label{}
  \sP = \op{S} (V) \otimes \wedge^2 V \otimes V 
\end{equation}
is more familiar.
As a $GL_2$-module the continuous dual is 
\begin{equation}\label{}
  \sP^* = \Hat{\op{S}} (V^*) \otimes \wedge^2 V^* \otimes V^* .
\end{equation}
Tracing the explicit module structure computed in the lemma, we see that this is the following tensor $\lie{w}_2$-module 
\begin{equation}\label{}
  \sP^* = \Omega^2 \otimes_{\C\llbracket z^\1,z^\2\rrbracket} \Omega^1 .
\end{equation}
Here, $\Omega^i$ denotes the $\C\llbracket z^\1,z^\2\rrbracket$-module of $i$-forms on the formal $2$-disk.

\subsection{Differential operators and the Weyl algebra}\label{s:weyl}

We will also make use of the dg algebra of polynomial differential operators on the formal, punctured $d$-disk.
Since the construction is very similar, we will be brief.
Let $\sD$ denote the sheaf of differential operators on $\AA^{d}-\{0\}$.
We denote the Jouanolou model of $\sD$:
\begin{equation}\label{}
  \cD^{poly}_d \define \bJ_{\mathring{\AA}^d} (\sD) .
\end{equation}
It is an explicit dg algebra model for $\R \Gamma(\mathring{\AA}^d, \sD)$.
As such, it is naturally equipped with the structure of a dg algebra.
We also have its completed version that we denote $\cD_{d}$.
Note that as a $\cJ_d$-module one has 
\begin{equation}\label{}
  \cD_d = \cJ_d \otimes \C[\del_\1,\dots,\del_\ddd] .
\end{equation}
The natural map $\lie{witt}_d \to \cD_d$ is a map of dg Lie algebras.
The $\dbar$-cohomology of $\cD_d$ is concentrated in degrees zero and $d-1$.

Let us turn to dimension $d=2$.
One has $H^0(\cD_d) = \fD_2$, the algebra of differential operators on the formal $2$-disk.
Thus $\fD_2 \cong \C\llbracket z^\1,z^\2\rrbracket [\del_\1,\del_\2]$ as $\C\llbracket z^\1,z^\2\rrbracket$-module.
We denote 
\begin{equation}\label{}
  \til \sP \define H^{1}_{\dbar}(\cD_2) .
\end{equation}
As a $GL_2$-module one has 
\begin{equation}\label{}
  \til \sP = \op{S}(V) \otimes \op{S}(\wedge^2 V \otimes V) .
\end{equation}
Its dual $\til \sP^*$ is the $\fD_2$-module coinduced from the (pro)tensor $GL_2$-module $\wedge^2 V^* \op{S}(V^*)$.
In other words
\begin{equation}\label{}
  \til \sP^* = \Omega^2 \otimes_{\C\llbracket z^\1,z^\2\rrbracket}\op{Sym}_{\C\llbracket z^\1,z^\2\rrbracket}  \left(\Omega^1\right) 
\end{equation}
where $\Omega^i$ denotes $i$-forms on the formal punctured disk and the symmetric algebra is over the ring of formal
power series.

Finally, we introduce the dg Weyl algebra associated to the higher-dimensional punctured disk.
This is a rather formal tool that we will use to port over the results of \cite{FFS} to prove the local, universal
version of the Grothendieck--Riemann--Roch theorem.
As a cochain complex the dg Weyl algebra $\cW_d$ is the same as $\cD_d$.
For the Weyl algebra we use the notations $z^\ii \leftrightarrow q^\ii$ and $\del_\ii \leftrightarrow p^{\ii}$.
The multiplication is the Moyal
$\star$-product:
\begin{equation}\label{}
  O' \star O'' \define m \circ \op{exp} \left(\frac12 (\del_p \otimes \del_q - \del_q \otimes \del_p ) \right) (O' \otimes
  O'') .
\end{equation}

\begin{lemma}
  There is a dg algebra isomorphism 
  \begin{equation}\label{}
    \sigma \colon \cD_d \xto{\cong} \cW_d
  \end{equation}
  defined as the formal symbol map 
  \begin{equation}\label{}
\sigma\left(f(z^{\1},...,z^{\dd})F(\partial_{\1},...,\partial_{\dd})\right) \define e^{-\frac{1}{2}\sum^d\limits_{\s=1}\partial_{p^{\s}}\partial_{q^{\s}}}(f(q^{\1},...,q^{\dd})F(p^{\1},...,p^{\dd})).
  \end{equation}
\end{lemma}

\section{Universal characteristic classes}\label{s:chern}

In the context of chiral conformal field theory the Virasoro class in $H^2_{Lie}(\lie{witt}_1)$ Witt algebra is an
anomaly. 
Specifically, it represents the determinant line bundle of the corresponding complex of Cauchy--Riemann, or
$\dbar$-operators, over the moduli of Riemann surfaces.
In this section, we extend this relationship between characteristic classes and central extensions of the Witt algebra
to arbitrary dimension.
We refer to \cite{KapranovNotes} for further elaboration of the connection between characteristic classes and central
extensions of the Witt algebra.

We use these universal characteristic classes to construct explicit representatives for the cohomology $\HH^2_{Lie}(\lie{witt}_d)$, and hence explicit models for the resulting central
extensions of the $d$-dimensional Witt algebra.

Before getting into the detailed results and constructions, we briefly sketch the form of the two cocycles which give rise
to the two independent cohomology classes of $\lie{witt}_2$ advertised in the first paragraph.
In order to express these formulas succinctly, we recall that there is a natural residue map defined on the Jouanolou model we consider, which one can
identify differential geometrically as an integral over the real codimension one sphere in affine space (we recall this
in detail below).
The first cocycle representative has the form:
\begin{equation}\label{eq:ch13}
  (T_1^\ii \del_\ii,T_2^\jj \del_\jj,T_3^\kkk \del_\kkk) \mapsto \op{Res} \left(\del_\ii T_1^\ii \del \del_\jj T_2^\jj \del \del_\kkk T_3^\kkk\right)
\end{equation}
The other cocycle representative is similar; it takes the form:
\begin{equation}\label{eq:ch1ch2}
  (T_1^\ii \del_\ii,T_2^\jj \del_\jj,T_3^\kkk \del_\kkk) \mapsto \frac16 \sum_{\tau} (\pm) \op{Res} \left(\del_\ii T_{\tau(1)}
    ^\ii \del \del_\jj T_{\tau(2)}^\kkk \del \del_\kkk
  T_{\tau(3)}^\jj\right).
\end{equation}
In both of these expressions the sums over $\ii,\jj,\kkk=1,2$ are implicit.

\subsection{Cyclic cohomology and residue}
\label{s:res}
Recall that $\cJ_d$ is the Jouanolou model for the structure sheaf on punctured $d$-dimensional affine space.
On the commutative dg algebra $\cJ_d$, there is a natural $GL_d$-invariant, total degree $1$ cyclic cocycle which we
will denote by $\rho$.
It represents a nontrivial element in the cyclic cohomology of $\cJ_d$ \cite[1.5]{FHK}.
We recall its definition presently.

We consider $\rho$ as an explicit cocycle in Connes' (reduced) $\lambda$-complex $\br C_\bu^\lambda(\cJ_d)$, see
\cite{loday1984cyclic} for example.
The formula of the cocycle $\rho$ involves the residue
\begin{equation}\label{}
  \op{Res} \colon \Omega^d \otimes_{\sO} \cJ_d \to \C[-d+1]
\end{equation}
which is the unique $GL_d$-equivariant cochain map which satisfies $\op{Res}(P \d^d z) = 1$.
We emphasize that $\Omega^d \otimes_{\sO}\cJ_d$ is isomorphic, as a $\lie{witt}_d$-representation, to the Jouanolou model for the canonical bundle $\sK$ on punctured affine
space.
Define the cyclic cocycle~$\rho$ by the formula
\begin{equation}\label{eq:rho}
  \rho(f_0,\ldots,f_d) = \op{Res}(f_0 \del f_1 \cdots \del f_d) .
\end{equation}
We make the important observation that since $\cJ_d$ is commutative dg algebra the complex~$\br C^\lambda_\bu(\cJ_d)$ is,
itself, naturally a commutative dg algebra.
We recall it's product below.

As we have pointed out $\cJ_d$ is naturally a dg $\lie{witt}_d$-module.
This extends, by derivation, to a natural $\lie{witt}_d$-module structure on $\br C^\lambda_\bu(\cJ_d)$.
We will need the following lemma.

\begin{lemma}\label{lem:rho}
  The degree one cocycle $\rho \in \br C^\bu_\lambda(\cJ_d)^1$ is $\lie{witt}_d$-invariant.
\end{lemma}
\begin{proof}
  By definition,
    \begin{align*}
  &\mathrm{Res}\left(X \cdot (f_0 \del f_1 \cdots \del f_d)\right)\\
  &=\mathrm{Res}\left((\iota_X\partial f_0)\cdot\partial f_1\cdots\partial f_d))\right)+\sum^d_{i=1}\mathrm{Res}\left(f_0\cdots\partial (\iota_X\partial f_i)\cdots\partial f_d))\right)\\
  &=\mathrm{Res}\left((\iota_X\partial f_0)\cdot\partial f_1\cdots\partial f_d))\right)+\sum^d_{i=1}(-1)^{i-1}\mathrm{Res}\left(\partial f_0\cdots (\iota_X\partial f_i)\cdots\partial f_d))\right)\\
  &=\mathrm{Res}\left(\iota_X(\partial f_0\cdot\partial f_1\cdots\partial f_d)\right)=0.
    \end{align*}

\end{proof}

\subsection{Universal Chern classes}

The Atiyah class is a measure of curvature typically defined in algebro-geometric contexts.
There are formal, or cyclic versions \cite{Calaque2005,Caldararu2000}, which appear in the Hochschild--Kostant--Rosenberg isomorphism
\cite{CDSatiyah} as well as in the algebraic index theorem \cite{FFS}.
Our construction can be interpreted as a formal, universal version of such Atiyah classes.
(Really, we only make use of the traces of powers of this class, so we produce universal Chern characters.)

Concretely, to construct Lie algebra cohomology classes of $\lie{witt}_d$ with trivial coefficients we use two ingredients:
\begin{itemize}
  \item[(1)] The natural trace map from the cyclic cohomology of the dg algebra $\cJ_d$-valued matrices which uses the
    residue from above.
  \item[(2)] A universal cocycle
    \begin{equation}\label{}
      \lie{c} \in C^\bu(\lie{witt}_d ;\br C_\bu^\lambda(\lie{gl}_d(\cJ_d)))^1
    \end{equation}
    of total degree $+1$.
\end{itemize}

Consider the dg algebra $\lie{gl}_d(\cJ_d)$ of matrices valued in the commutative dg algebra~$\cJ_d$.
We will define a universal element $\til{\lie{c}}$ in the Chevalley--Eilenberg cochain complex of $\lie{witt}_d$ with coefficients in
$\br C_\bu^\lambda(\lie{gl}_d(\cJ_d))$.
This element is of total degree $+1$ and decomposes as
\begin{equation}\label{eq:c}
  \til{\lie{c}} = \til{\lie{c}}_1 + \til{\lie{c}}_2+ \cdots 
\end{equation}
where
\begin{equation}\label{}
  \til{\lie{c}}_{k+1} \colon \wedge^{k+1} \lie{witt}_d \to \br C_\bu^\lambda(\lie{gl}_d(\cJ_d))
\end{equation}
is defined by
\begin{align}
      \til{\lie{c}}_{k+1}(T_0\wedge\cdots\wedge T_{k}) &=\frac{(-1)^k}{2^k} \sum_{\permute \in S_k}(\pm)^{'} \cdot \left(J T_{0}, J T_{\permute(1)}, \ldots, J T_{\permute(k)}\right) \\
 & =\frac{(-1)^k}{2^k\cdot(k+1)} \sum_{\permute' \in S_{k+1}}(\pm)^{''}\cdot \left(J T_{\permute'(0)}, J T_{\permute'(1)}, \ldots, J T_{\permute'(k)}\right) 
\end{align}
where
$$
(\pm)'=\mathrm{sgn}(\permute)\cdot\varepsilon(\permute;T_1,\ldots,T_k) ,
$$
$$
(\pm)''=\mathrm{sgn}(\permute')\cdot\varepsilon(\permute';T_0,T_1,\ldots,T_k) ,
$$
and $J \colon \lie{witt}_d \to \lie{gl}_d(\cJ_d)$ is the \textit{Jacobian} matrix which to an element $T = \sum^d\limits_{\ii=1} f^\ii
\del_\ii$ assigns the $\cJ_d$-valued matrix
\begin{equation}\label{}
  (JT)_{\jj}^\kkk = \frac{\del f^\kkk}{\del z^\jj} \in \cJ_d .
\end{equation}
We observe that $\til{\lie{c}}$ is of total degree $+1$.

\begin{theorem}
  The element $\til{\lie{c}} = \sum\limits_{k\geq 0} \til{\lie{c}}_{k+1}=\til{\lie{c}}_1 + \til{\lie{c}}_2+ \cdots $ is a cocycle of total degree one.
\end{theorem}
\begin{proof}
    \begin{figure}[htbp]
        \centering

\tikzset{every picture/.style={line width=0.75pt}} 

\begin{tikzpicture}[x=0.75pt,y=0.75pt,yscale=-1,xscale=1]

\draw   (181,78) .. controls (181,53.7) and (200.7,34) .. (225,34) .. controls (249.3,34) and (269,53.7) .. (269,78) .. controls (269,102.3) and (249.3,122) .. (225,122) .. controls (200.7,122) and (181,102.3) .. (181,78) -- cycle ;
\draw [color={rgb, 255:red, 189; green, 16; blue, 224 }  ,draw opacity=1 ]   (218,35) ;
\draw [shift={(218,35)}, rotate = 0] [color={rgb, 255:red, 189; green, 16; blue, 224 }  ,draw opacity=1 ][fill={rgb, 255:red, 189; green, 16; blue, 224 }  ,fill opacity=1 ][line width=0.75]      (0, 0) circle [x radius= 3.35, y radius= 3.35]   ;
\draw [color={rgb, 255:red, 189; green, 16; blue, 224 }  ,draw opacity=1 ]   (250,43) ;
\draw [shift={(250,43)}, rotate = 0] [color={rgb, 255:red, 189; green, 16; blue, 224 }  ,draw opacity=1 ][fill={rgb, 255:red, 189; green, 16; blue, 224 }  ,fill opacity=1 ][line width=0.75]      (0, 0) circle [x radius= 3.35, y radius= 3.35]   ;
\draw [color={rgb, 255:red, 189; green, 16; blue, 224 }  ,draw opacity=1 ]   (266,95) ;
\draw [shift={(266,95)}, rotate = 0] [color={rgb, 255:red, 189; green, 16; blue, 224 }  ,draw opacity=1 ][fill={rgb, 255:red, 189; green, 16; blue, 224 }  ,fill opacity=1 ][line width=0.75]      (0, 0) circle [x radius= 3.35, y radius= 3.35]   ;
\draw [color={rgb, 255:red, 189; green, 16; blue, 224 }  ,draw opacity=1 ]   (206,117) ;
\draw [shift={(206,117)}, rotate = 0] [color={rgb, 255:red, 189; green, 16; blue, 224 }  ,draw opacity=1 ][fill={rgb, 255:red, 189; green, 16; blue, 224 }  ,fill opacity=1 ][line width=0.75]      (0, 0) circle [x radius= 3.35, y radius= 3.35]   ;
\draw [color={rgb, 255:red, 189; green, 16; blue, 224 }  ,draw opacity=1 ]   (256,110) ;
\draw [shift={(256,110)}, rotate = 0] [color={rgb, 255:red, 189; green, 16; blue, 224 }  ,draw opacity=1 ][fill={rgb, 255:red, 189; green, 16; blue, 224 }  ,fill opacity=1 ][line width=0.75]      (0, 0) circle [x radius= 3.35, y radius= 3.35]   ;
\draw    (247,122.91) .. controls (253.58,124.79) and (258.39,122.26) .. (265.59,115.37) ;
\draw [shift={(267,114)}, rotate = 135.31] [fill={rgb, 255:red, 0; green, 0; blue, 0 }  ][line width=0.08]  [draw opacity=0] (12,-3) -- (0,0) -- (12,3) -- cycle    ;
\draw    (277,88.91) .. controls (279.83,97.42) and (278.2,97.89) .. (272.95,107.2) ;
\draw [shift={(272,108.91)}, rotate = 298.61] [fill={rgb, 255:red, 0; green, 0; blue, 0 }  ][line width=0.08]  [draw opacity=0] (12,-3) -- (0,0) -- (12,3) -- cycle    ;
\draw   (178,237) .. controls (178,212.7) and (197.7,193) .. (222,193) .. controls (246.3,193) and (266,212.7) .. (266,237) .. controls (266,261.3) and (246.3,281) .. (222,281) .. controls (197.7,281) and (178,261.3) .. (178,237) -- cycle ;
\draw [color={rgb, 255:red, 189; green, 16; blue, 224 }  ,draw opacity=1 ]   (215,194) ;
\draw [shift={(215,194)}, rotate = 0] [color={rgb, 255:red, 189; green, 16; blue, 224 }  ,draw opacity=1 ][fill={rgb, 255:red, 189; green, 16; blue, 224 }  ,fill opacity=1 ][line width=0.75]      (0, 0) circle [x radius= 3.35, y radius= 3.35]   ;
\draw [color={rgb, 255:red, 189; green, 16; blue, 224 }  ,draw opacity=1 ]   (247,202) ;
\draw [shift={(247,202)}, rotate = 0] [color={rgb, 255:red, 189; green, 16; blue, 224 }  ,draw opacity=1 ][fill={rgb, 255:red, 189; green, 16; blue, 224 }  ,fill opacity=1 ][line width=0.75]      (0, 0) circle [x radius= 3.35, y radius= 3.35]   ;
\draw [color={rgb, 255:red, 189; green, 16; blue, 224 }  ,draw opacity=1 ]   (263,254) ;
\draw [shift={(263,254)}, rotate = 0] [color={rgb, 255:red, 189; green, 16; blue, 224 }  ,draw opacity=1 ][fill={rgb, 255:red, 189; green, 16; blue, 224 }  ,fill opacity=1 ][line width=0.75]      (0, 0) circle [x radius= 3.35, y radius= 3.35]   ;
\draw [color={rgb, 255:red, 189; green, 16; blue, 224 }  ,draw opacity=1 ]   (203,276) ;
\draw [shift={(203,276)}, rotate = 0] [color={rgb, 255:red, 189; green, 16; blue, 224 }  ,draw opacity=1 ][fill={rgb, 255:red, 189; green, 16; blue, 224 }  ,fill opacity=1 ][line width=0.75]      (0, 0) circle [x radius= 3.35, y radius= 3.35]   ;
\draw [color={rgb, 255:red, 189; green, 16; blue, 224 }  ,draw opacity=1 ]   (253,269) ;
\draw [shift={(253,269)}, rotate = 0] [color={rgb, 255:red, 189; green, 16; blue, 224 }  ,draw opacity=1 ][fill={rgb, 255:red, 189; green, 16; blue, 224 }  ,fill opacity=1 ][line width=0.75]      (0, 0) circle [x radius= 3.35, y radius= 3.35]   ;
\draw    (244,281.91) .. controls (250.58,283.79) and (255.39,281.26) .. (262.59,274.37) ;
\draw [shift={(264,273)}, rotate = 135.31] [fill={rgb, 255:red, 0; green, 0; blue, 0 }  ][line width=0.08]  [draw opacity=0] (12,-3) -- (0,0) -- (12,3) -- cycle    ;
\draw    (274,247.91) .. controls (276.83,256.42) and (275.2,256.89) .. (269.95,266.2) ;
\draw [shift={(269,267.91)}, rotate = 298.61] [fill={rgb, 255:red, 0; green, 0; blue, 0 }  ][line width=0.08]  [draw opacity=0] (12,-3) -- (0,0) -- (12,3) -- cycle    ;
\draw   (62,402) .. controls (62,377.7) and (81.7,358) .. (106,358) .. controls (130.3,358) and (150,377.7) .. (150,402) .. controls (150,426.3) and (130.3,446) .. (106,446) .. controls (81.7,446) and (62,426.3) .. (62,402) -- cycle ;
\draw [color={rgb, 255:red, 189; green, 16; blue, 224 }  ,draw opacity=1 ]   (99,359) ;
\draw [shift={(99,359)}, rotate = 0] [color={rgb, 255:red, 189; green, 16; blue, 224 }  ,draw opacity=1 ][fill={rgb, 255:red, 189; green, 16; blue, 224 }  ,fill opacity=1 ][line width=0.75]      (0, 0) circle [x radius= 3.35, y radius= 3.35]   ;
\draw [color={rgb, 255:red, 189; green, 16; blue, 224 }  ,draw opacity=1 ]   (131,367) ;
\draw [shift={(131,367)}, rotate = 0] [color={rgb, 255:red, 189; green, 16; blue, 224 }  ,draw opacity=1 ][fill={rgb, 255:red, 189; green, 16; blue, 224 }  ,fill opacity=1 ][line width=0.75]      (0, 0) circle [x radius= 3.35, y radius= 3.35]   ;
\draw [color={rgb, 255:red, 189; green, 16; blue, 224 }  ,draw opacity=1 ]   (147,419) ;
\draw [shift={(147,419)}, rotate = 0] [color={rgb, 255:red, 189; green, 16; blue, 224 }  ,draw opacity=1 ][fill={rgb, 255:red, 189; green, 16; blue, 224 }  ,fill opacity=1 ][line width=0.75]      (0, 0) circle [x radius= 3.35, y radius= 3.35]   ;
\draw [color={rgb, 255:red, 189; green, 16; blue, 224 }  ,draw opacity=1 ]   (87,441) ;
\draw [shift={(87,441)}, rotate = 0] [color={rgb, 255:red, 189; green, 16; blue, 224 }  ,draw opacity=1 ][fill={rgb, 255:red, 189; green, 16; blue, 224 }  ,fill opacity=1 ][line width=0.75]      (0, 0) circle [x radius= 3.35, y radius= 3.35]   ;
\draw [color={rgb, 255:red, 189; green, 16; blue, 224 }  ,draw opacity=1 ]   (137,434) ;
\draw [shift={(137,434)}, rotate = 0] [color={rgb, 255:red, 189; green, 16; blue, 224 }  ,draw opacity=1 ][fill={rgb, 255:red, 189; green, 16; blue, 224 }  ,fill opacity=1 ][line width=0.75]      (0, 0) circle [x radius= 3.35, y radius= 3.35]   ;
\draw    (158,412.91) .. controls (160.84,421.42) and (159.2,421.89) .. (153.95,431.2) ;
\draw [shift={(153,432.91)}, rotate = 298.61] [fill={rgb, 255:red, 0; green, 0; blue, 0 }  ][line width=0.08]  [draw opacity=0] (12,-3) -- (0,0) -- (12,3) -- cycle    ;
\draw   (243,402) .. controls (243,377.7) and (262.7,358) .. (287,358) .. controls (311.3,358) and (331,377.7) .. (331,402) .. controls (331,426.3) and (311.3,446) .. (287,446) .. controls (262.7,446) and (243,426.3) .. (243,402) -- cycle ;
\draw [color={rgb, 255:red, 189; green, 16; blue, 224 }  ,draw opacity=1 ]   (280,359) ;
\draw [shift={(280,359)}, rotate = 0] [color={rgb, 255:red, 189; green, 16; blue, 224 }  ,draw opacity=1 ][fill={rgb, 255:red, 189; green, 16; blue, 224 }  ,fill opacity=1 ][line width=0.75]      (0, 0) circle [x radius= 3.35, y radius= 3.35]   ;
\draw [color={rgb, 255:red, 189; green, 16; blue, 224 }  ,draw opacity=1 ]   (312,367) ;
\draw [shift={(312,367)}, rotate = 0] [color={rgb, 255:red, 189; green, 16; blue, 224 }  ,draw opacity=1 ][fill={rgb, 255:red, 189; green, 16; blue, 224 }  ,fill opacity=1 ][line width=0.75]      (0, 0) circle [x radius= 3.35, y radius= 3.35]   ;
\draw [color={rgb, 255:red, 189; green, 16; blue, 224 }  ,draw opacity=1 ]   (328,419) ;
\draw [shift={(328,419)}, rotate = 0] [color={rgb, 255:red, 189; green, 16; blue, 224 }  ,draw opacity=1 ][fill={rgb, 255:red, 189; green, 16; blue, 224 }  ,fill opacity=1 ][line width=0.75]      (0, 0) circle [x radius= 3.35, y radius= 3.35]   ;
\draw [color={rgb, 255:red, 189; green, 16; blue, 224 }  ,draw opacity=1 ]   (268,441) ;
\draw [shift={(268,441)}, rotate = 0] [color={rgb, 255:red, 189; green, 16; blue, 224 }  ,draw opacity=1 ][fill={rgb, 255:red, 189; green, 16; blue, 224 }  ,fill opacity=1 ][line width=0.75]      (0, 0) circle [x radius= 3.35, y radius= 3.35]   ;
\draw [color={rgb, 255:red, 189; green, 16; blue, 224 }  ,draw opacity=1 ]   (318,434) ;
\draw [shift={(318,434)}, rotate = 0] [color={rgb, 255:red, 189; green, 16; blue, 224 }  ,draw opacity=1 ][fill={rgb, 255:red, 189; green, 16; blue, 224 }  ,fill opacity=1 ][line width=0.75]      (0, 0) circle [x radius= 3.35, y radius= 3.35]   ;
\draw    (309,446.91) .. controls (315.58,448.79) and (320.39,446.26) .. (327.59,439.37) ;
\draw [shift={(329,438)}, rotate = 135.31] [fill={rgb, 255:red, 0; green, 0; blue, 0 }  ][line width=0.08]  [draw opacity=0] (12,-3) -- (0,0) -- (12,3) -- cycle    ;
\draw [color={rgb, 255:red, 189; green, 16; blue, 224 }  ,draw opacity=1 ]   (190,51) ;
\draw [shift={(190,51)}, rotate = 0] [color={rgb, 255:red, 189; green, 16; blue, 224 }  ,draw opacity=1 ][fill={rgb, 255:red, 189; green, 16; blue, 224 }  ,fill opacity=1 ][line width=0.75]      (0, 0) circle [x radius= 3.35, y radius= 3.35]   ;
\draw [color={rgb, 255:red, 189; green, 16; blue, 224 }  ,draw opacity=1 ]   (189,208) ;
\draw [shift={(189,208)}, rotate = 0] [color={rgb, 255:red, 189; green, 16; blue, 224 }  ,draw opacity=1 ][fill={rgb, 255:red, 189; green, 16; blue, 224 }  ,fill opacity=1 ][line width=0.75]      (0, 0) circle [x radius= 3.35, y radius= 3.35]   ;
\draw [color={rgb, 255:red, 189; green, 16; blue, 224 }  ,draw opacity=1 ]   (72,374) ;
\draw [shift={(72,374)}, rotate = 0] [color={rgb, 255:red, 189; green, 16; blue, 224 }  ,draw opacity=1 ][fill={rgb, 255:red, 189; green, 16; blue, 224 }  ,fill opacity=1 ][line width=0.75]      (0, 0) circle [x radius= 3.35, y radius= 3.35]   ;
\draw [color={rgb, 255:red, 189; green, 16; blue, 224 }  ,draw opacity=1 ]   (253,374) ;
\draw [shift={(253,374)}, rotate = 0] [color={rgb, 255:red, 189; green, 16; blue, 224 }  ,draw opacity=1 ][fill={rgb, 255:red, 189; green, 16; blue, 224 }  ,fill opacity=1 ][line width=0.75]      (0, 0) circle [x radius= 3.35, y radius= 3.35]   ;

\draw (83,68.4) node [anchor=north west][inner sep=0.75pt]    {$b\til{\lie{c}}_{k+1} =$};
\draw (201,11.4) node [anchor=north west][inner sep=0.75pt]  [font=\scriptsize]  {$JT_{\permute ( 1)}$};
\draw (143,54.4) node [anchor=north west][inner sep=0.75pt]    {$\sum $};
\draw (251,22.4) node [anchor=north west][inner sep=0.75pt]  [font=\scriptsize]  {$JT_{\permute ( 2)}$};
\draw (174,121.4) node [anchor=north west][inner sep=0.75pt]  [font=\scriptsize]  {$JT_{\permute ( k)}$};
\draw (281,115.4) node [anchor=north west][inner sep=0.75pt]  [font=\scriptsize]  {$[ JT_{\permute ( i)} ,JT_{\permute ( j)}]$};
\draw (315,50.4) node [anchor=north west][inner sep=0.75pt]  [font=\small]  {$([ JT_{\permute ( i)} ,JT_{\permute ( j)}])_{\s}^{\rrr} =\frac{\partial T_{\permute ( i)}^{\rrr}}{\partial z^{\ttt}}\frac{\partial T_{\permute ( j)}^{\ttt}}{\partial z^{\s}} -\frac{\partial T_{\permute ( j)}^{\rrr}}{\partial z^{\ttt}}\frac{\partial T_{\permute ( i)}^{\ttt}}{\partial z^{\s}}$};
\draw (52,224.4) node [anchor=north west][inner sep=0.75pt]    {$\d^{\rm Lie}\til{\lie{c}}_{k} =$};
\draw (199,170.4) node [anchor=north west][inner sep=0.75pt]  [font=\scriptsize]  {$JT_{\permute ( 1)}$};
\draw (140,213.4) node [anchor=north west][inner sep=0.75pt]    {$\sum $};
\draw (248,181.4) node [anchor=north west][inner sep=0.75pt]  [font=\scriptsize]  {$JT_{\permute ( 2)}$};
\draw (171,280.4) node [anchor=north west][inner sep=0.75pt]  [font=\scriptsize]  {$JT_{\permute ( k)}$};
\draw (278,274.4) node [anchor=north west][inner sep=0.75pt]  [font=\scriptsize]  {$J[ T_{\permute ( i)} ,T_{\permute ( j)}]$};
\draw (322,197.4) node [anchor=north west][inner sep=0.75pt]  [font=\small]  {$( J[ T_{\permute ( i)} ,T_{\permute ( j)}])_{\s}^{\rrr} =\frac{\partial \left( T_{\permute ( i)}^{\ttt}\frac{\partial T_{\permute ( j)}^{\rrr}}{\partial z^{\ttt}} -T_{\permute ( j)}^{\ttt}\frac{\partial T_{\permute ( i)}^{\rrr}}{\partial z^{\ttt}}\right)}{\partial z^{\s}}$};
\draw (83,335.4) node [anchor=north west][inner sep=0.75pt]  [font=\scriptsize]  {$JT_{\permute ( 1)}$};
\draw (24,378.4) node [anchor=north west][inner sep=0.75pt]    {$+\sum $};
\draw (132,346.4) node [anchor=north west][inner sep=0.75pt]  [font=\scriptsize]  {$JT_{\permute ( 2)}$};
\draw (55,445.4) node [anchor=north west][inner sep=0.75pt]  [font=\scriptsize]  {$JT_{\permute ( k)}$};
\draw (139,437.4) node [anchor=north west][inner sep=0.75pt]  [font=\scriptsize]  {$T_{\permute ( i)}( JT_{\permute ( j)})$};
\draw (264,335.4) node [anchor=north west][inner sep=0.75pt]  [font=\scriptsize]  {$JT_{\permute ( 1)}$};
\draw (198,378.4) node [anchor=north west][inner sep=0.75pt]    {$+\sum $};
\draw (313,346.4) node [anchor=north west][inner sep=0.75pt]  [font=\scriptsize]  {$JT_{\permute ( 2)}$};
\draw (236,445.4) node [anchor=north west][inner sep=0.75pt]  [font=\scriptsize]  {$JT_{\permute ( k)}$};
\draw (330,422.4) node [anchor=north west][inner sep=0.75pt]  [font=\scriptsize]  {$T_{\permute ( j)}( JT_{\permute ( i)})$};
\draw (389,337.4) node [anchor=north west][inner sep=0.75pt]  [font=\small]  {$ \begin{array}{l}
T_{\permute ( i)}( JT_{\permute ( j)}) -T_{\permute ( j)}( JT_{\permute ( i)})\\
=T_{\permute ( i)}^{\ttt} \partial _{z^{\ttt}}\left(\frac{\partial T_{\permute ( j)}^{\rrr}}{\partial z^{\s}}\right) -T_{\permute ( j)}^{\ttt} \partial _{z^{\ttt}}\left(\frac{\partial T_{\permute ( i)}^{\rrr}}{\partial z^{\s}}\right)
\end{array}$};
\draw (155,29.4) node [anchor=north west][inner sep=0.75pt]  [font=\scriptsize]  {$JT_{\permute ( 0)}$};
\draw (153,190.4) node [anchor=north west][inner sep=0.75pt]  [font=\scriptsize]  {$JT_{\permute ( 0)}$};
\draw (38,356.4) node [anchor=north west][inner sep=0.75pt]  [font=\scriptsize]  {$JT_{\permute ( 0)}$};
\draw (219,356.4) node [anchor=north west][inner sep=0.75pt]  [font=\scriptsize]  {$JT_{\permute ( 0)}$};

\end{tikzpicture}
        \caption{$b\til{\lie{c}}_{k+1}$ and $\d^{\rm Lie}\til{\lie{c}}_{k}$.}
        \label{fig:bAndLie}
    \end{figure}
We follow the fig \ref{fig:bAndLie}. 
\begin{align*}
&b    \til{\lie{c}}_{k+1}(T_0\wedge\cdots\wedge T_{k})\\
& = \frac{(-1)^{k+1}}{2^{k+1}}\sum_{\permute \in S_k}(\pm)_{'}\cdot  b\left(J T_{0}, J T_{\permute(1)}, \ldots, J T_{\permute(k)}\right)\\
&=\frac{(-1)^{k+1}\cdot 2^{-k}}{(k+1)k}\sum_{\permute' \in S_{k+1}}\sum_{0\leq i<j\leq k}(\pm)_{''}\cdot \left([JT_{\permute'(i)},JT_{\permute'(j)}],JT_{\permute'(0)},\dots,\widehat{JT_{\permute'(i)}},\dots,\widehat{JT_{\permute'(j)}},\dots, JT_{\permute'(k)}\right)
\end{align*}
where
$$
(\pm)_{'}=\mathrm{sgn}(\permute)\cdot\varepsilon(\permute;T_1,\ldots,T_k) ,
$$
$$
(\pm)_{''}=\mathrm{sgn}(\permute')\cdot\varepsilon(\permute';T_0,T_1,\ldots,T_k)\cdot (-1)^{i+j-1}\cdot (-1)^{|T_{\permute'(i)}|\cdot\sum^{i-1}\limits_{p=0}|T_{\permute'(p)}|+|T_{\permute'(j)}|\cdot\sum^{j-1}\limits_{q=0,q\neq i}|T_{\permute'(q)}|} .
$$
On the other hand 
\begin{align*}
&\mathbf{d}^{\mathrm{Lie}}\til{\lie{c}}_{k}\left(T_0\wedge\cdots\wedge T_k\right)\\&=\til{\lie{c}}_{k}\left(\sum_{0\leq i<j\leq k}(\pm)_{\flat'}\cdot [T_i,T_j]\wedge T_0\wedge\cdots\wedge\widehat{T_i}\wedge\cdots\wedge\widehat{T_j}\wedge\cdots\wedge T_k\right)\\
&+\sum_{0\leq i\leq k}(\pm)_{\flat''}\cdot T_i\til{\lie{c}}_{k}\left(T_0\wedge\cdots\wedge\widehat{T_i}\wedge\cdots\wedge T_k\right)\\
&=(-2)^{-k}\sum_{0\leq i<j\leq k}\sum_{\permute\in S_{k-1}}(\pm)_{\natural'}\cdot\left(J[T_i,T_j], JT_{\permute(0)},\ldots,\widehat{JT_i},\ldots,\widehat{JT_j},\ldots, JT_{\permute(k)}\right)\\
&+\frac{(-2)^{-k}}{k}\sum_{0\leq i\leq k}\sum_{\til{\permute}\in S_{k}}(\pm)_{\natural''}\cdot T_i\left(JT_{\til{\permute}(0)},JT_{\til{\permute}(1)},\ldots,\widehat{JT_i},\ldots, JT_{\til{\permute}(k)}\right)\\
&=\frac{(-2)^{-k}}{(k+1)k}\sum_{\permute' \in S_{k+1}}\sum_{0\leq i<j\leq k}(\pm)_{'}\cdot\left(J[T_{\permute'(i)},T_{\permute'(j)}],JT_{\permute'(0)},\dots,\widehat{JT_{\permute'(i)}},\dots,\widehat{JT_{\permute'(j)}},\dots, JT_{\permute'(k)}\right)- \\
&\frac{(-2)^{-k}}{(k+1)k}\sum_{\permute' \in S_{k+1}}\sum_{0\leq i<j\leq k}(\pm)_{''}\cdot\left(T_{\permute'(i)}(JT_{\permute'(j)})-T_{\permute'(j)}(JT_{\permute'(i)}),JT_{\permute'(0)},\dots,\widehat{JT_{\permute'(i)}},\dots,\widehat{JT_{\permute'(j)}},\dots, JT_{\permute'(k)}\right),
\end{align*}
where
\[
(\pm)_{\flat'}=(-1)^{i+j-1}\cdot (-1)^{|T_{i}|\cdot\sum^{i-1}\limits_{p=0}|T_{p}|+|T_{j}|\cdot\sum^{j-1}\limits_{q=0,q\neq i}|T_{q}|},\quad (\pm)_{\flat''}=(-1)^{i+1}\cdot (-1)^{|T_{i}|\cdot\sum^{i-1}\limits_{p=0}|T_{p}|},
\]
\[
(\pm)_{\natural'}=\mathrm{sgn}(\permute)\cdot\varepsilon(\permute;T_0,\ldots,\widehat{T_i},\ldots,\widehat{T_j},\ldots,T_k)\cdot (-1)^{i+j-1}\cdot (-1)^{|T_{i}|\cdot\sum^{i-1}\limits_{p=0}|T_{\permute(p)}|+|T_{j}|\cdot\sum^{j-1}\limits_{q=0,q\neq i}|T_{\permute(q)}|},
\]
\[
(\pm)_{\natural''}=\mathrm{sgn}(\til{\permute})\cdot\varepsilon(\til{\permute};T_0,\ldots,\widehat{T_i},\ldots,T_k)\cdot (-1)^{i+1}\cdot (-1)^{|T_{i}|\cdot\sum^{i-1}\limits_{p=0}|T_{p}|}. 
\]
We compute
$$
([ JT_{\permute' ( i)} ,JT_{\permute' ( j)}])_{\s}^{\rrr} =\sum^d\limits_{\ttt=1}\frac{\partial T_{\permute' ( i)}^{\rrr}}{\partial z^{\ttt}}\frac{\partial T_{\permute' ( j)}^{\ttt}}{\partial z^{\s}} -\sum^d\limits_{\ttt=1}\frac{\partial T_{\permute' ( j)}^{\rrr}}{\partial z^{\ttt}}\frac{\partial T_{\permute' ( i)}^{\ttt}}{\partial z^{\s}},
$$
$$
( J[ T_{\permute' ( i)} ,T_{\permute '( j)}])_{\s}^{\rrr} =\sum^d\limits_{\ttt=1}\frac{\partial \left( T_{\permute' ( i)}^{\ttt}\frac{\partial T_{\permute' ( j)}^{\rrr}}{\partial z^{\ttt}} -T_{\permute' ( j)}^{\ttt}\frac{\partial T_{\permute' ( i)}^{\rrr}}{\partial z^{\ttt}}\right)}{\partial z^{\s}}.
$$
$$
\begin{array}{l}
\left(T_{\permute' ( i)}( JT_{\permute' ( j)}) -T_{\permute '( j)}( JT_{\permute' ( i)})\right)^\rrr_\s\\
=\sum^d\limits_{\ttt=1} T_{\permute '( i)}^{\ttt} \partial_{z^{\ttt}}\left(\frac{\partial T_{\permute '( j)}^{\rrr}}{\partial z^{\s}}\right) -\sum^d\limits_{\ttt=1}T_{\permute' ( j)}^{\ttt} \partial _{z^{\ttt}}\left(\frac{\partial T_{\permute' ( i)}^{\rrr}}{\partial z^{\s}}\right).
\end{array}
$$
\end{proof}

For any dg algebra $A$ the matrix trace $\op{Tr} \colon \lie{gl}_d(A) \to A$ extends to a cochain map of
$\lambda$-complexes
\begin{equation}\label{}
  \op{Tr} \colon \br C_\bu^\lambda(\lie{gl}_d(A)) \xto{\simeq} \br C_\bu^\lambda (A)
\end{equation}
where it is a quasi-isomorphism \cite{loday1984cyclic}.
The map $\op{Tr}$ is explicitly defined as 
\begin{equation}\label{}
  \op{Tr}( (M_0 \otimes a_0), \ldots, (M_\ell \otimes a_\ell)) \define \op{Tr}(M_0 \cdots
  M_\ell)\cdot (a_0, \ldots, a_\ell) .
\end{equation}

Returning to $A = \cJ_d$, we let $\lie{c}$ be the trace of this cochain $\til{c}$:
\begin{equation}\label{}
  \lie{c} = \op{Tr}(\til{\lie{c}}) \in C^\bu(\lie{witt}_d \, ; \, \br C_\bu^\lambda(\cJ_d))^1.
\end{equation}
\begin{corollary}
  The element $\lie{c} = \sum \lie{c}_k$ is a cocycle of total degree one.
\end{corollary}

For $A$ a commutative dg algebra, Connes' complex $\overline{C}_\bu^\lambda(A)$ admits, itself, the natural structure of a
commutative dg algebra \cite[proposition 3.7]{loday1984cyclic}. 
It is defined by the formula
\begin{equation}\label{eqn:starproduct}
x*y \define x\times(By),
\end{equation}
where
\[
B(a_0,a_1,\dots,a_d)=\sum^n_{i=0}\pm (1,a_i,\dots,a_n,a_0,\dots,a_{i-1}),
\]
and 
\[
\pm=(-1)^{i\cdot n}\cdot\varepsilon(\permute_i;a_0,a_1,\dots,a_d),\quad \permute_i(a_0,a_1,\dots,a_d)=(a_i,\dots,a_n,a_0,\dots,a_{i-1})
\]
The operation $\times$ is the so-called shuffle product, given by 
\[
(a,a_1,\dots,a_p)\times (a',a_{p+1},\dots,a_{p+q})=\sum_{\text{Shuffle}}(\pm)\cdot\mathrm{sgn}(\permute)\cdot (aa',a_{\permute^{-1}(1)},\dots,a_{\permute^{-1}(p+q)}),
\]
\[
(\pm)=(-1)^{|a'|\sum^p\limits_{i=1}|a_i|}\cdot\varepsilon(\permute_i;a_1,\dots,a_{p+q})
\]
where the sum is over all permutations $\permute$ of $\{1,\dots,p+q\}$ such that $\permute(1)<\cdots<\permute(p)$ and $\permute(p+1)<\cdots<\permute(p+q)$.
Intuitively, this product involves putting two circles together with a cyclic twist; see figure \ref{fig:CyclicProd}. A circle with marked points can be viewed as an element in $\overline{A}^{\otimes \bullet+1}/(1-t).$
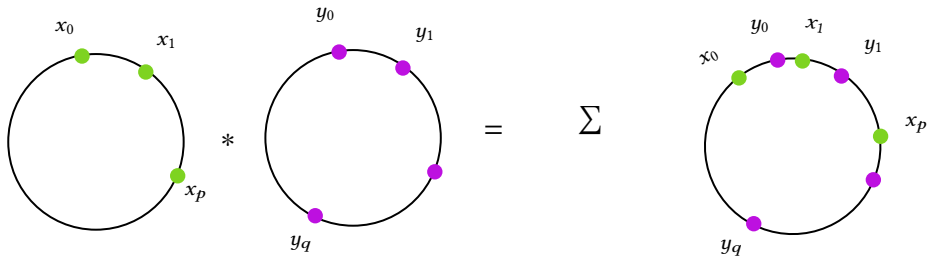
\begin{figure}[htbp]
    \centering

\tikzset{every picture/.style={line width=0.75pt}} 

\begin{tikzpicture}[x=0.75pt,y=0.75pt,yscale=-1,xscale=1]

\draw   (89,107) .. controls (89,82.7) and (108.7,63) .. (133,63) .. controls (157.3,63) and (177,82.7) .. (177,107) .. controls (177,131.3) and (157.3,151) .. (133,151) .. controls (108.7,151) and (89,131.3) .. (89,107) -- cycle ;
\draw [color={rgb, 255:red, 126; green, 211; blue, 33 }  ,draw opacity=1 ]   (126,64) ;
\draw [shift={(126,64)}, rotate = 0] [color={rgb, 255:red, 126; green, 211; blue, 33 }  ,draw opacity=1 ][fill={rgb, 255:red, 126; green, 211; blue, 33 }  ,fill opacity=1 ][line width=0.75]      (0, 0) circle [x radius= 3.35, y radius= 3.35]   ;
\draw [color={rgb, 255:red, 126; green, 211; blue, 33 }  ,draw opacity=1 ]   (158,72) ;
\draw [shift={(158,72)}, rotate = 0] [color={rgb, 255:red, 126; green, 211; blue, 33 }  ,draw opacity=1 ][fill={rgb, 255:red, 126; green, 211; blue, 33 }  ,fill opacity=1 ][line width=0.75]      (0, 0) circle [x radius= 3.35, y radius= 3.35]   ;
\draw [color={rgb, 255:red, 126; green, 211; blue, 33 }  ,draw opacity=1 ]   (174,124) ;
\draw [shift={(174,124)}, rotate = 0] [color={rgb, 255:red, 126; green, 211; blue, 33 }  ,draw opacity=1 ][fill={rgb, 255:red, 126; green, 211; blue, 33 }  ,fill opacity=1 ][line width=0.75]      (0, 0) circle [x radius= 3.35, y radius= 3.35]   ;
\draw   (218,105) .. controls (218,80.7) and (237.7,61) .. (262,61) .. controls (286.3,61) and (306,80.7) .. (306,105) .. controls (306,129.3) and (286.3,149) .. (262,149) .. controls (237.7,149) and (218,129.3) .. (218,105) -- cycle ;
\draw [color={rgb, 255:red, 189; green, 16; blue, 224 }  ,draw opacity=1 ]   (255,62) ;
\draw [shift={(255,62)}, rotate = 0] [color={rgb, 255:red, 189; green, 16; blue, 224 }  ,draw opacity=1 ][fill={rgb, 255:red, 189; green, 16; blue, 224 }  ,fill opacity=1 ][line width=0.75]      (0, 0) circle [x radius= 3.35, y radius= 3.35]   ;
\draw [color={rgb, 255:red, 189; green, 16; blue, 224 }  ,draw opacity=1 ]   (287,70) ;
\draw [shift={(287,70)}, rotate = 0] [color={rgb, 255:red, 189; green, 16; blue, 224 }  ,draw opacity=1 ][fill={rgb, 255:red, 189; green, 16; blue, 224 }  ,fill opacity=1 ][line width=0.75]      (0, 0) circle [x radius= 3.35, y radius= 3.35]   ;
\draw [color={rgb, 255:red, 189; green, 16; blue, 224 }  ,draw opacity=1 ]   (303,122) ;
\draw [shift={(303,122)}, rotate = 0] [color={rgb, 255:red, 189; green, 16; blue, 224 }  ,draw opacity=1 ][fill={rgb, 255:red, 189; green, 16; blue, 224 }  ,fill opacity=1 ][line width=0.75]      (0, 0) circle [x radius= 3.35, y radius= 3.35]   ;
\draw [color={rgb, 255:red, 189; green, 16; blue, 224 }  ,draw opacity=1 ]   (243,144) ;
\draw [shift={(243,144)}, rotate = 0] [color={rgb, 255:red, 189; green, 16; blue, 224 }  ,draw opacity=1 ][fill={rgb, 255:red, 189; green, 16; blue, 224 }  ,fill opacity=1 ][line width=0.75]      (0, 0) circle [x radius= 3.35, y radius= 3.35]   ;
\draw   (444.04,130.49) .. controls (432.27,109.23) and (439.97,82.45) .. (461.23,70.69) .. controls (482.49,58.92) and (509.27,66.62) .. (521.04,87.88) .. controls (532.8,109.14) and (525.1,135.92) .. (503.84,147.68) .. controls (482.58,159.45) and (455.81,151.75) .. (444.04,130.49) -- cycle ;
\draw [color={rgb, 255:red, 126; green, 211; blue, 33 }  ,draw opacity=1 ]   (455.59,74.95) ;
\draw [shift={(455.59,74.95)}, rotate = 0] [color={rgb, 255:red, 126; green, 211; blue, 33 }  ,draw opacity=1 ][fill={rgb, 255:red, 126; green, 211; blue, 33 }  ,fill opacity=1 ][line width=0.75]      (0, 0) circle [x radius= 3.35, y radius= 3.35]   ;
\draw [color={rgb, 255:red, 126; green, 211; blue, 33 }  ,draw opacity=1 ]   (487.46,66.46) ;
\draw [shift={(487.46,66.46)}, rotate = 0] [color={rgb, 255:red, 126; green, 211; blue, 33 }  ,draw opacity=1 ][fill={rgb, 255:red, 126; green, 211; blue, 33 }  ,fill opacity=1 ][line width=0.75]      (0, 0) circle [x radius= 3.35, y radius= 3.35]   ;
\draw [color={rgb, 255:red, 126; green, 211; blue, 33 }  ,draw opacity=1 ]   (526.64,104.21) ;
\draw [shift={(526.64,104.21)}, rotate = 0] [color={rgb, 255:red, 126; green, 211; blue, 33 }  ,draw opacity=1 ][fill={rgb, 255:red, 126; green, 211; blue, 33 }  ,fill opacity=1 ][line width=0.75]      (0, 0) circle [x radius= 3.35, y radius= 3.35]   ;
\draw [color={rgb, 255:red, 189; green, 16; blue, 224 }  ,draw opacity=1 ]   (475,66) ;
\draw [shift={(475,66)}, rotate = 0] [color={rgb, 255:red, 189; green, 16; blue, 224 }  ,draw opacity=1 ][fill={rgb, 255:red, 189; green, 16; blue, 224 }  ,fill opacity=1 ][line width=0.75]      (0, 0) circle [x radius= 3.35, y radius= 3.35]   ;
\draw [color={rgb, 255:red, 189; green, 16; blue, 224 }  ,draw opacity=1 ]   (507,74) ;
\draw [shift={(507,74)}, rotate = 0] [color={rgb, 255:red, 189; green, 16; blue, 224 }  ,draw opacity=1 ][fill={rgb, 255:red, 189; green, 16; blue, 224 }  ,fill opacity=1 ][line width=0.75]      (0, 0) circle [x radius= 3.35, y radius= 3.35]   ;
\draw [color={rgb, 255:red, 189; green, 16; blue, 224 }  ,draw opacity=1 ]   (523,126) ;
\draw [shift={(523,126)}, rotate = 0] [color={rgb, 255:red, 189; green, 16; blue, 224 }  ,draw opacity=1 ][fill={rgb, 255:red, 189; green, 16; blue, 224 }  ,fill opacity=1 ][line width=0.75]      (0, 0) circle [x radius= 3.35, y radius= 3.35]   ;
\draw [color={rgb, 255:red, 189; green, 16; blue, 224 }  ,draw opacity=1 ]   (463,148) ;
\draw [shift={(463,148)}, rotate = 0] [color={rgb, 255:red, 189; green, 16; blue, 224 }  ,draw opacity=1 ][fill={rgb, 255:red, 189; green, 16; blue, 224 }  ,fill opacity=1 ][line width=0.75]      (0, 0) circle [x radius= 3.35, y radius= 3.35]   ;

\draw (194,101.4) node [anchor=north west][inner sep=0.75pt]    {$*$};
\draw (111,45.4) node [anchor=north west][inner sep=0.75pt]  [font=\scriptsize]  {$x_{0}$};
\draw (162,52.4) node [anchor=north west][inner sep=0.75pt]  [font=\scriptsize]  {$x_{1}$};
\draw (241,38.4) node [anchor=north west][inner sep=0.75pt]  [font=\scriptsize]  {$y_{0}$};
\draw (291,48.4) node [anchor=north west][inner sep=0.75pt]  [font=\scriptsize]  {$y_{1}$};
\draw (176,127.4) node [anchor=north west][inner sep=0.75pt]  [font=\scriptsize]  {$x_{p}$};
\draw (228,152.4) node [anchor=north west][inner sep=0.75pt]  [font=\scriptsize]  {$y_{q}$};
\draw (326,97.4) node [anchor=north west][inner sep=0.75pt]    {$=$};
\draw (432.74,63.68) node [anchor=north west][inner sep=0.75pt]  [font=\scriptsize,rotate=-331.04]  {$x_{0}$};
\draw (488.45,43.22) node [anchor=north west][inner sep=0.75pt]  [font=\scriptsize,rotate=-10.15]  {$x_{1}$};
\draw (537.49,93.12) node [anchor=north west][inner sep=0.75pt]  [font=\scriptsize,rotate=-358.47]  {$x_{p}$};
\draw (459,44.4) node [anchor=north west][inner sep=0.75pt]  [font=\scriptsize]  {$y_{0}$};
\draw (516,54.4) node [anchor=north west][inner sep=0.75pt]  [font=\scriptsize]  {$y_{1}$};
\draw (373,87.4) node [anchor=north west][inner sep=0.75pt]    {$\sum $};
\draw (444,154.4) node [anchor=north west][inner sep=0.75pt]  [font=\scriptsize]  {$y_{q}$};

\end{tikzpicture}
    \caption{The $*$ product on Connes' $\lambda$-complex}
    \label{fig:CyclicProd}
\end{figure}

We are prepared to construct classes in $\mathbb{H}^{2}_{\mathrm{Lie}}(\lie{witt}_d)$.
\begin{definition}\label{dfn:chern}
    For non-negative integers $i_1,i_2,\cdots,i_{d+1}$ such that $\sum^{d+1}\limits_{p=1}p\cdot i_p=d+1$ we define
    \begin{equation}\label{eq:chernwitt}
    (\mathrm{ch}_1)^{i_1}(\mathrm{ch}_2)^{i_2}\cdots (\mathrm{ch}_{d+1})^{i_{d+1}}\in C^\bu(\lie{witt}_d ; \C )
  \end{equation}
    by 
    \begin{equation}\label{eq:chernwitt2}
      \rho\left(\op{Tr}(\lie{c}_1)^{*i_1}*\op{Tr}(\lie{c}_2)^{*i_2}*\cdots*\op{Tr}(\lie{c}_{d+1})^{*i_{d+1}}\right).
    \end{equation}
\end{definition}

\begin{remark}
  The element $\op{ch}_{d+1}$ is a homogenous algebraic combination of lower chern characters $\op{ch}_1, \ldots, \op{ch}_{d}$.
The proof of this when $d=2$ is direct calculation; we will include the $d>2$ proof in forthcoming work.
\end{remark}

%

In the remainder of this subsection we will prove that the cochain \eqref{eq:chernwitt} is indeed a cocycle.
To prepare, we prove some lemmas involving the product $*$ on (reduced) cyclic chains.

\begin{lemma}\label{Lem:KillbOp}
  The map $\rho$ from \eqref{eq:rho} is the composition 
  \begin{equation}
    \overline{C}_d^\lambda(\cJ_d)\xrightarrow{\pi}\frac{\Omega^d\otimes_{\mathcal{O}}\cJ_d}{\mathrm{Im}(\del)}\xrightarrow{\mathrm{Res}}\C
  \end{equation}
  where $\pi$ is $\pi(f_0,\ldots,f_d) = [f_0 \del f_1 \cdots \del f_d]$.
  For any cyclic chains $v,w$ one has 
    $$
    \pi\left(b(v)*w\right)=0,\forall v,w.
    $$
\end{lemma}
\begin{proof}
    We first prove that $\pi$ is well defined. 
    The permutation operator on Connes' cyclic chains is defined by
    $$
    t(f_0,f_1,\dots,f_d)=(-1)^d\cdot (f_d,f_0,\dots,f_{d-1}).
    $$
    We have
    \begin{align*}
      & \pi\left((1-t)(f_0,f_1,\dots,f_d)\right)\\
      &=f_0\partial f_1\cdots \partial f_d-(-1)^df_d\partial f_0\cdots \partial f_{d-1}\\
      &=(-1)^{d-1}\partial\left(f_0f_d\cdot \partial f_1\cdots \partial f_{d-1}\right).
    \end{align*}
    Thus $\pi$ maps $(1-t)a$ to $\mathrm{Im}(\del)$ which induce a map $\overline{\cJ}^{\otimes d+1}_d/(1-t)\rightarrow
    \frac{\Omega^d\otimes_{\mathcal{O}}\cJ_d}{\mathrm{Im}(\del)}$.

    We compute
    \begin{align*}
       &b(v)*w\\
       &=(Bb(v_0,v_1,\dots,v_i))\times (w_0,w_1,\dots,w_{d-i-1})\\
       &=\sum_{\permute}\sum_{\tau}(-1)^{\chi}\permute(w_0,\tau(bv),w_1,\dots,w_{d-i-1})
    \end{align*}
    where summation is over cyclic permutation $\tau$ and shuffle $\permute$, and $(-1)^{\chi}$ is a sign factor that is not important in the proof. Note that $\pi$ is invariant under the shuffle $\permute$ action, we only need to prove that
    $$
    \partial^{\otimes i}\left(b(v_0,v_1,\dots,v_i)\right)=0.
    $$
We compute as follows
\begin{align*}
    &\partial^{\otimes i}\left(b(v_0,v_1,\dots,v_i)\right)\\
    &=\partial^{\otimes i}\left(\sum^{i-1}_{k=0}(-1)^k(v_0,\dots,v_kv_{k+1},\dots,v_i)+(-1)^i(v_iv_0,v_1,\dots,v_{i-1})\right)\\
     &=\sum^{i-1}_{k=0}(-1)^k\partial v_0\cdots\partial(v_kv_{k+1})\cdots\partial v_i+(-1)^i\partial(v_iv_0)\cdot \partial v_1\cdots\partial v_{i-1}\\
     &=v_0\cdot\partial v_1\cdots\partial v_i+(-1)^iv_0\partial v_i\cdot \partial v_1\cdots\partial v_{i-1}=0.
\end{align*}
\end{proof}


\begin{theorem}\label{thm:weakchern}
  The cochain \eqref{eq:chernwitt} is closed for the total differential on $C^{\bu}_{\text{Lie}} (\lie{witt}_d;\C)$ and
  so represents a class in $\HH^2_{\text{Lie}}(\lie{witt}_d)$.
\end{theorem}

\begin{proof}
$$
(\d^{\rm Lie}+\dbar)\rho\left(\op{Tr}(\lie{c}_1)^{*i_1}*\op{Tr}(\lie{c}_2)^{*i_2}*\cdots*\op{Tr}(\lie{c}_d)^{*i_d}\right)(\xi)
    $$    
    $$
    =\mathrm{Res}\circ \pi\left(\op{Tr}(\lie{c}_1)^{*i_1}*\op{Tr}(\lie{c}_2)^{*i_2}*\cdots*\op{Tr}(\lie{c}_d)^{*i_d}
    \right)((\d^{\rm Lie}+\dbar)\xi)
    $$
     $$
    =\mathrm{Res}\circ \pi\left(\op{Tr}(\lie{c}_1)^{*i_1}*\op{Tr}(\lie{c}_2)^{*i_2}*\cdots*\op{Tr}(\lie{c}_d)^{*i_d}
    (\d^{\rm Lie}\xi)\right)\quad (\text{Res annihilates }\ \dbar(-))
    $$
    Using Lemma \ref{lem:rho}, we have
    \begin{align*}
        &\mathrm{Res}\circ \pi\left(\op{Tr}(\lie{c}_1)^{*i_1}*\op{Tr}(\lie{c}_2)^{*i_2}*\cdots*\op{Tr}(\lie{c}_d)^{*i_d}(\d^{\rm Lie}\xi)\right)\\
        &=\mathrm{Res}\circ \pi \left(\d^{\rm Lie}\left(\op{Tr}(\lie{c}_1)^{*i_1}*\op{Tr}(\lie{c}_2)^{*i_2}*\cdots*\op{Tr}(\lie{c}_d)^{*i_d}\right)(\xi)\right)\\
        &=\mathrm{Res}\circ \pi \left(\left(\sum bx* y\right)(\xi)\right)=0.
    \end{align*}
\end{proof}

\subsection{Some low-dimensional examples}

When $d=1$ the only class the above construction produces is the standard Virasoro cocycle defined, up to scale, by: 
\begin{equation}\label{}
  \op{ch}_1^2 \colon \left(f \del_z, g \del_z \right) \mapsto \op{Res} \left(\del_z f \del_z^2 g \d z \right) .
\end{equation}
Thus, in this case we find that our construction recovers the well-known isomorphism $H^2(\lie{witt}_1) = \C$.

When $d=2$ we have two cocycles corresponding to the polynomials, $\op{ch}_1^3, \op{ch}_1 \op{ch}_2$.
Explicitly:
\begin{equation}\label{}
  \op{ch}_1^3 \colon (T_1 , T_2, T_3) \mapsto \op{Res}\left(\op{Tr}(J T_1) \del \op{Tr}(J T_2) \del \op{Tr}(J T_3)
    \right) . 
\end{equation}
This is the cocycle we advertised in equation \eqref{eq:ch13}.
Likewise, the other cocycle in dimension two is:
\begin{equation}\label{}
  \op{ch}_1 \op{ch}_2 \colon (T_1 , T_2, T_3) \mapsto \frac16 \sum_{\tau} (\pm)_{\tau}'' \op{Res}\left(\op{Tr}(J T_{\tau(1)} \op{Tr}(\del J
    T_{\tau(2)} \del J T_{\tau(3)})
    \right) .
\end{equation}
This is the cocycle we advertised in \eqref{eq:ch1ch2}.

\section{Cohomology of $\lie{witt}_2$}\label{s:2dgf}

In this section we classify central extensions of the two-dimensional dg Witt algebra $\lie{witt}_2$.
Previously, we have constructed classes in $\HH^2_{Lie}(\lie{witt}_d)$ from universal characteristic classes.
We have not yet shown that these classes represent non-trivial elements in cohomology nor did we show that these
classes are linearly independent in cohomology.
We tie up these loose ends completely in the case $d=2$.
When $d=2$ there are two universal characteristic classes $\op{ch}
_1^3$ (see \eqref{eq:ch13}) and $\op{ch}_1 \op{ch}_2$ (see \eqref{eq:ch1ch2}) which give rise to central extensions.
The section splits up into two main parts.

First, we show that the classes $\op{ch}_1^3$ and $\op{ch}_1 \op{ch}_2$ are indeed independent in cohomology by
direct calculation.
Thus $\dim \HH^2_{Lie} (\lie{witt}_2) \geq 2$.
Next, we use the Hochschild--Serre spectral sequence to obtain an upper bound thus finishing the proof of the theorem.

\begin{theorem}[]
\label{thm:2dgf} 
\begin{itemize}
  \item[(1)] The map
    \[
      \C[\op{ch}_1,\op{ch}_2]_6 \to \HH^2(\lie{witt}_2)
    \]
    defined in definition \ref{dfn:chern} is
    injective.
  \item[(2)] One has 
\begin{equation}
  \dim \HH^2_{Lie}(\lie{witt}_2) = 2 .
\end{equation}
\end{itemize}
\end{theorem}

Together, parts (1) and (2) of this theorem imply theorem \ref{thm:chern}.
In the final part of this section we prove that the second Lie algebra cohomology of Lie algebra of (derived) differential operators
on the punctured $2$-disk is one-dimensional.
This fact will be used in our proof of the local, universal Grothendieck--Riemann--Roch theorem in section \ref{s:2dGRR}
.

\subsection{Injectivity}

In this section we consider only dimension $d=2$.
Theorem \ref{thm:weakchern} presents two universal classes in $\HH^2_{Lie}(\lie{witt}_2)$ which we denote $\op{ch}_1^3$ and $\op{ch}_1 \op{ch}_2$.
Our immediate goal is to show that these classes are independent (in particular, nontrivial).

We consider dual Chevalley--Eilenberg chains, which are simply elements of the graded symmetric algebra on $\lie{witt}
_2[1]$.
Our strategy is to find, for example, a chain $\mathbf{X}\in \bigwedge^{\bullet}\lie{witt}_{2}$ of total degree
$(-2)$\footnote{Note that our cohomologically graded Chevalley--Eilenberg complex is concentrated in non-positive
degrees.}
which detects the nontriviality of $\op{ch}_1^3$ at the
cochain level:
\[
(\dbar+d^{\mathrm{Lie}})\mathbf{X}=0,\quad (\mathrm{ch}_1)^3(\mathbf{X})\neq 0.
\]
We will do the same for $\op{ch}_1 \op{ch}_2$ as well.

We need the following straightforward lemma.
\begin{lemma}
Let
\[
X_1^{\mathrm{I}} = P\,\del_\1,\quad
X_2^{\mathrm{I}} = (z^\1)^3\del_\1,\quad
X_3^{\mathrm{I}} = (z^\2)^2\del_\2
\]
\[
X_1^{\mathrm{II}} = P\,\del_\1,\quad
X_2^{\mathrm{II}} = (z^\1)^2\del_\1,\quad
X_3^{\mathrm{II}} = z^\1 (z^\2)^2\del_\2
\]
\[
X_1^{\mathrm{III}} = P\,\del_\1,\quad
X_2^{\mathrm{III}} = z^\1 z^\2\del_\1,\quad
X_3^{\mathrm{III}} = (z^\1)^2 z^\2\del_\2
\]
\[
X_1^{\mathrm{IV}} = P\,\del_\1,\quad
X_2^{\mathrm{IV}} = (z^\1)^2 z^\2\del_\1,\quad
X_3^{\mathrm{IV}} = z^\1 z^\2\del_\2
\]
and let 
\[
\begin{aligned}
Y^{\mathrm{I}}_{12}
&:= 4 z^\1 x^\2\,\del_\1 + (z^\1)^2 x^\1 x^\2\,\del_\1,\\[4pt]
Y^{\mathrm{I}}_{13}
&:= -\,z^\2 x^\1 x^\2\,\del_\1 - x^\1\,\del_\1,
\end{aligned}
\qquad
\begin{aligned}
Y^{\mathrm{II}}_{12}
&:= 3 x^\2\,\del_\1 + z^\1 x^\1 x^\2\,\del_\1,\\[4pt]
Y^{\mathrm{II}}_{13}
&:= -\,z^\2 x^\1\,\del_\2 + (z^\2)^2 (x^\2)^2\,\del_\1,
\end{aligned}
\]
\[
\begin{aligned}
Y^{\mathrm{III}}_{12}
&:= -x^\1\,\del_\1 - z^\1 (x^\1)^2\,\del_\1,\\[4pt]
Y^{\mathrm{III}}_{13}
&:= 2 z^\2 x^\2\,\del_\2 + z^\1 z^\2 (x^\2)^2\,\del_\1,
\end{aligned}
\qquad
\begin{aligned}
Y^{\mathrm{IV}}_{12}
&:= 2 x^\2 z^\2\,\del_\1 - (z^\1)^2 (x^\1)^2\,\del_\1,\\[4pt]
Y^{\mathrm{IV}}_{13}
&:= -x^\1\,\del_\2 + z^\2 (x^\2)^2\,\del_\1.
\end{aligned}
\]

We have
\begin{enumerate}
\item
\[
[X_2^{\mathrm{I}},X_3^{\mathrm{I}}]=0,\qquad
[X_2^{\mathrm{II}},X_3^{\mathrm{II}}]-[X_2^{\mathrm{III}},X_3^{\mathrm{III}}]+[X_2^{\mathrm{IV}},X_3^{\mathrm{IV}}]=0,
\]
\item
for each choice $?\in\{\mathrm{I},\mathrm{II},\mathrm{III},\mathrm{IV}\}$ and $i\in\{2,3\}$,
\[
[X_1^{?},X_i^{?}]=\dbar Y^{?}_{1i}.
\]
\end{enumerate}
\end{lemma}

\begin{proof}
The first part is a direct computation:
\[
[X_2^{\mathrm{I}},X_3^{\mathrm{I}}]=0,
\]
\[
[X_2^{\mathrm{II}},X_3^{\mathrm{II}}]
= \bigl[(z^\1)^2\del_\1,\, z^\1 (z^\2)^2\del_\2\bigr]
= (z^\1)^2 (z^\2)^2 \,\del_\2,
\]
\[
[X_2^{\mathrm{III}},X_3^{\mathrm{III}}]
= \bigl[z^\1 z^\2\del_\1,\, (z^\1)^2 z^\2\del_\2\bigr]
= 2 (z^\1)^2 (z^\2)^2 \,\del_\2 - (z^\1)^3 z^\2 \,\del_\1,
\]
\[
[X_2^{\mathrm{IV}},X_3^{\mathrm{IV}}]
= \bigl[(z^\1)^2 z^\2\del_\1,\, z^\1 z^\2\del_\2\bigr]
= (z^\1)^2 (z^\2)^2 \,\del_\2 - (z^\1)^3 z^\2 \,\del_\1.
\]

To prove (b) we use the following observations (each bracket is rewritten as a $\dbar$--exact expression):

\[
\begin{aligned}
[X_1^{\mathrm{I}},X_2^{\mathrm{I}}]
&= \dbar\Bigl(3 (z^\1)^2 \frac{x^\2}{z^\1}\,\del_\1
- (z^\1)^3 \del_\1\bigl(\tfrac{x^\2}{z^\1}\bigr)\,\del_\1\Bigr),
\\[4pt]
[X_1^{\mathrm{I}},X_3^{\mathrm{I}}]
&= \dbar\Bigl((z^\2)^2\del_\2\bigl(\tfrac{x^\1}{z^\2}\bigr)\,\del_\1\Bigr),
\end{aligned}
\]

\[
\begin{aligned}
[X_1^{\mathrm{II}},X_2^{\mathrm{II}}]
&=2 z^\1 P\,\del_\1 - (z^\1)^2 \del_\1P\,\del_\1\\
&= \dbar\Bigl(2 z^\1 \frac{x^\2}{z^\1}\,\del_\1
- (z^\1)^2 \del_\1\bigl(\tfrac{x^\2}{z^\1}\bigr)\,\del_\1\Bigr)\\
&= \dbar\bigl(3 x^\2\,\del_\1 + z^\1 x^\1 x^\2\,\del_\1\bigr),
\end{aligned}
\]

\[
\begin{aligned}
[X_1^{\mathrm{II}},X_3^{\mathrm{II}}]
&= (z^\2)^2 P\,\del_\2 - z^\1 (z^\2)^2 \del_\2P\,\del_\1\\
&= \dbar\Bigl((z^\2)^2 \bigl(-\tfrac{x^\1}{z^\2}\bigr)\,\del_\2
- z^\1 (z^\2)^2 \del_\2\bigl(\tfrac{x^\2}{z^\1}\bigr)\,\del_\1\Bigr)\\
&= \dbar\bigl(-z^\2 x^\1\,\del_\2 + (z^\2)^2 (x^\2)^2\,\del_\1\bigr),
\end{aligned}
\]

\[
\begin{aligned}
[X_1^{\mathrm{III}},X_2^{\mathrm{III}}]
&= z^\2 P\,\del_\1 - z^\1 z^\2 \del_\1P\,\del_\1\\
&= \dbar\Bigl(z^\2 \bigl(-\tfrac{x^\1}{z^\2}\bigr)\,\del_\1
- z^\1 z^\2 \del_\1\bigl(-\tfrac{x^\1}{z^\2}\bigr)\,\del_\1\Bigr)\\
&= \dbar\bigl(-x^\1\,\del_\1 - z^\1 (x^\1)^2\,\del_\1\bigr),
\end{aligned}
\]

\[
\begin{aligned}
[X_1^{\mathrm{III}},X_3^{\mathrm{III}}]
&= 2 z^\2 z^\1 P\,\del_\2 - (z^\1)^2 z^\2 \del_\2P\,\del_\1\\
&= \dbar\Bigl(2 z^\2 z^\1 \tfrac{x^\2}{z^\1}\,\del_\2
- (z^\1)^2 z^\2 \del_\2\bigl(\tfrac{x^\2}{z^\1}\bigr)\,\del_\1\Bigr)\\
&= \dbar\bigl(2 z^\2 x^\2\,\del_\2 + z^\1 z^\2 (x^\2)^2\,\del_\1\bigr),
\end{aligned}
\]

\[
\begin{aligned}
[X_1^{\mathrm{IV}},X_2^{\mathrm{IV}}]
&= 2 z^\1 z^\2 P\,\del_\1 - (z^\1)^2 z^\2 \del_\1P\,\del_\1\\
&= \dbar\Bigl(2 z^\1 z^\2 \tfrac{x^\2}{z^\1}\,\del_\1
- (z^\1)^2 z^\2 \del_\1\bigl(-\tfrac{x^\1}{z^\2}\bigr)\,\del_\1\Bigr)\\
&= \dbar\bigl(2 x^\2 z^\2\,\del_\1 - (z^\1)^2 (x^\1)^2\,\del_\1\bigr),
\end{aligned}
\]

\[
\begin{aligned}
[X_1^{\mathrm{IV}},X_3^{\mathrm{IV}}]
&= z^\2 P\,\del_\2 - z^\1 z^\2 \del_\2P\,\del_\1\\
&= \dbar\Bigl(z^\2 \bigl(-\tfrac{x^\1}{z^\2}\bigr)\,\del_\2
- z^\1 z^\2 \del_\2\bigl(\tfrac{x^\2}{z^\1}\bigr)\,\del_\1\Bigr)\\
&= \dbar\bigl(-x^\1\,\del_\2 + z^\2 (x^\2)^2\,\del_\1\bigr).
\end{aligned}
\]

These identities establish (b).
\end{proof}

\medskip

Define
\[
\mathbb{X}:= X_1^{\mathrm{I}}\wedge X_2^{\mathrm{I}}\wedge X_3^{\mathrm{I}},
\qquad
\tilde{\mathbb{X}}:= X_1^{\mathrm{II}}\wedge X_2^{\mathrm{II}}\wedge X_3^{\mathrm{II}}
- X_1^{\mathrm{III}}\wedge X_2^{\mathrm{III}}\wedge X_3^{\mathrm{III}}
+ X_1^{\mathrm{IV}}\wedge X_2^{\mathrm{IV}}\wedge X_3^{\mathrm{IV}}.
\]

\begin{prop}
We can extend $\mathbb{X},\tilde{\mathbb{X}}\in \bigwedge^{3}\mathfrak{witt}_{2}$ to
$\mathbf{X},\tilde{\mathbf{X}}\in \bigwedge^{\bullet}\mathfrak{witt}_{2}$ such that
\[
(\dbar+d^{\mathrm{Lie}})\mathbf{X}=(\dbar+d^{\mathrm{Lie}})\tilde{\mathbf{X}}=0
\]
and
\[
\det\begin{pmatrix}
(\mathrm{ch}_1)^3(\mathbf{X}) & (\mathrm{ch}_1\cdot \mathrm{ch}_2)(\mathbf{X})\\[4pt]
(\mathrm{ch}_1)^3(\tilde{\mathbf{X}}) & (\mathrm{ch}_1\cdot \mathrm{ch}_2)(\tilde{\mathbf{X}})
\end{pmatrix}\neq 0.
\]
\end{prop}

\begin{proof}
Set
\[
\mathbf{X}
= \mathbb{X}
- \bigl(Y_{12}^{\mathrm{I}}\wedge X_3^{\mathrm{I}} - Y_{13}^{\mathrm{I}}\wedge X_2^{\mathrm{I}}\bigr)
+ 2\,\del_\1\wedge (z^\1)^2\del_\1,
\]
\[
\begin{aligned}
\tilde{\mathbf{X}}
&= \tilde{\mathbb{X}}
- \bigl(Y_{12}^{\mathrm{II}}\wedge X_3^{\mathrm{II}} - Y_{13}^{\mathrm{II}}\wedge X_2^{\mathrm{II}}\bigr)
+ \bigl(Y_{12}^{\mathrm{III}}\wedge X_3^{\mathrm{III}} - Y_{13}^{\mathrm{III}}\wedge X_2^{\mathrm{III}}\bigr)\\
&\quad - \bigl(Y_{12}^{\mathrm{IV}}\wedge X_3^{\mathrm{IV}} - Y_{13}^{\mathrm{IV}}\wedge X_2^{\mathrm{IV}}\bigr)
+ 2\,\del_\1\wedge (z^\1)^2\del_\1 - 4\,\del_\2\wedge (z^\2)^2\del_\2.
\end{aligned}
\]

A direct evaluation of the Chern-character pairings gives:
\[
(\mathrm{ch}_1)^3(\mathbf{X}) = (\mathrm{ch}_1)^3\bigl(X_1^{\mathrm{I}}\wedge X_2^{\mathrm{I}}\wedge X_3^{\mathrm{I}}\bigr)
= -3!\cdot 2!,
\]
\[
(\mathrm{ch}_1\cdot \mathrm{ch}_2)(\mathbf{X})
= (\mathrm{ch}_1\cdot \mathrm{ch}_2)\bigl(X_1^{\mathrm{I}}\wedge X_2^{\mathrm{I}}\wedge X_3^{\mathrm{I}}\bigr)
= 3!\cdot 2!,
\]
and for the other summands one finds
\[
\begin{gathered}
(\mathrm{ch}_1)^3\bigl(X_1^{\mathrm{II}}\wedge X_2^{\mathrm{II}}\wedge X_3^{\mathrm{II}}\bigr) = -4, \quad
(\mathrm{ch}_1\cdot \mathrm{ch}_2)\bigl(X_1^{\mathrm{II}}\wedge X_2^{\mathrm{II}}\wedge X_3^{\mathrm{II}}\bigr)=8,\\[4pt]
(\mathrm{ch}_1)^3\bigl(X_1^{\mathrm{III}}\wedge X_2^{\mathrm{III}}\wedge X_3^{\mathrm{III}}\bigr) = 2, \quad
(\mathrm{ch}_1\cdot \mathrm{ch}_2)\bigl(X_1^{\mathrm{III}}\wedge X_2^{\mathrm{III}}\wedge X_3^{\mathrm{III}}\bigr) = -6,\\[4pt]
(\mathrm{ch}_1)^3\bigl(X_1^{\mathrm{IV}}\wedge X_2^{\mathrm{IV}}\wedge X_3^{\mathrm{IV}}\bigr) = 2, \quad
(\mathrm{ch}_1\cdot \mathrm{ch}_2)\bigl(X_1^{\mathrm{IV}}\wedge X_2^{\mathrm{IV}}\wedge X_3^{\mathrm{IV}}\bigr) = -2,
\end{gathered}
\]
which yields
\[
(\mathrm{ch}_1)^3(\tilde{\mathbf{X}}) = -4,\qquad
(\mathrm{ch}_1\cdot \mathrm{ch}_2)(\tilde{\mathbf{X}}) = 12.
\]
This completes the proof.
\end{proof}

We have shown.

\begin{corollary}
\begin{equation}
\dim \mathbb{H}^2_{\mathrm{Lie}}\left(\lie{witt}_{2}\right)\geq 2.
\end{equation}
\end{corollary}

This already indicates a departure from the $d=1$ case where central extensions of $\lie{witt}_1 = \C(\!(z)\!) \del_z$ are
unique up to isomorphism; the unique central extension $H^2_{Lie}(\lie{witt}_1) = \C$ being the famous Virasoro.
In the remainder of this section we prove that the above inequality is an equality.

\subsection{Second cohomology of $\lie{witt}_2$}

Given any (non-negatively graded) dg Lie algebra $\lie{g}$, whose internal differential is denoted $\dbar$, there is a spectral sequence 
\begin{equation}\label{}
  H_{Lie}^\bu (H^\bu_{\dbar}(\lie{g})) \; \implies \; \HH_{Lie}^\bu(\lie{g})  
\end{equation}
where the second page is the Lie algebra cohomology of the \textit{graded} Lie algebra $H^\bu_{\dbar}(\lie{g})$---the
one we get from taking the internal cohomology of $\lie{g}$.
The spectral sequence converges to the Lie algebra cohomology of the dg Lie algebra.

\begin{lemma}\label{lem:sswitt}
    The cohomology $\mathbb{H}^{\bullet}_{\mathrm{Lie}}\left(\lie{witt}_{2}\right)$ is computed by a spectral sequence
    whose second page can be identified with the Lie algebra 
    cohomology of formal vector fields $\lie{w}_2$ on the $2$-disk with coefficients in a particular $\lie{w}_2$-module:
    \begin{equation}\label{}
      H_{Lie}^\bu\left(\lie{w}_2 \, ; \, \op{S}^\bu(\Omega^1 \otimes_{\C\llbracket z^1, z^2\rrbracket} \Omega^2)\right) \implies \HH_{Lie}^\bu(\lie{witt}_2).
    \end{equation}
    Here, $\op{S}^\bu(-)$ is the ordinary (not graded) symmetric algebra functor.
    This spectral sequence respects the ring structure.
\end{lemma}
\begin{proof}
  Consider the filtration on the Chevalley--Eilenberg complex $C^\bu(\lie{witt}_2 ; \C)$ by graded symmetric degree.
  The resulting spectral sequence has $E_1$-page 
  \begin{equation}\label{}
    C^\bu  \left( H^\bu (\lie{witt}_2) \right) .
  \end{equation}
  This is the Chevalley--Eilenberg complex of the graded Lie algebra $H^\bu(\lie{witt}_2)$.
  We have described the (internal) cohomology of $\lie{witt}_2$ in proposition \ref{prop:dbarwitt}.
  It is a semi-direct product graded Lie algebra concentrated in degrees zero and one:
  \begin{equation}\label{}
    H^\bu(\lie{witt}_2) \simeq \lie{w}_2 \ltimes \sP[-1]
  \end{equation}
  The $\lie{w}_2$-module structure on $\sP = H^1_{\dbar}(\lie{witt}_2)$ is presented in proposition \ref{prop:dbarwitt}. 

  The graded vector space underlying the $E_1$-page is therfore
  \begin{align*}
    \op{Sym}\left(\left(H^\bu_{\dbar}(\lie{witt}_2)[1]\right)^* \right) & \simeq \op{Sym}\left((\lie{w}_2[1])^*\right)
    \otimes \op{Sym}(\sP^*) .
  \end{align*}
  Frurthermore, as $\lie{w}_2$-modules, $\sP^*$ is the tensor module $\Omega^2 \otimes_{\C\llbracket z^1,z^1\rrbracket} \Omega^1$.
  Thus, the $E_1$-page is of the form $C_{Lie}^\bu(\lie{w}_2 \, ; \, \Omega^2 \otimes_{\C\llbracket z^1 , z^2\rrbracket}
  \Omega^1)$.
  The $E_1$-differential is readily seen to be the Gelfand--Fuks differential with coefficients.
  The assertion follows.
\end{proof}

As a corollary of this result, and the form of the spectral sequence, we see that 
\begin{equation}\label{eq:decomp}
  \dim \HH^2_{Lie}(\lie{witt}_2)  \leq \dim \left(H^2(\lie{w}_2) \oplus H^2(\lie{witt}_2 ; \Omega^1 \otimes \Omega^2) \oplus \bigoplus_{j >
  1} H^2(\lie{w}_2 ; \op{S}^j (\Omega^1 \otimes_{\sO} \Omega^2) \right) .
\end{equation}
It is a classic result of Gelfand and Fuks that $H^j_{Lie}(\lie{w}_2) = 0$ for $j \leq 4$ \cite{FuksBook}.
In particular, the first summand on the right hand side of \eqref{eq:decomp} vanishes.

The last summand is also zero, as we will now prove.
We find it easier to use Lie algebra \textit{homology} rather than cohomology.
We observe that 
\begin{equation}\label{}
  H^2_{Lie} (\lie{w}_2 ; \sM) = \left(H_2^{Lie} (\lie{w}_2^{poly} ; \sM^*)\right)^* 
\end{equation}
for any module $\sM$.
Thus, to show that the last term in \eqref{eq:decomp} vanishes it suffices to prove:
\begin{lemma}
  For $j>1$ one has
  \[
    H^{Lie}_2\left(\lie{w}^{poly}_2;\op{S}^j(\sP) \right)=0 
  \]
\end{lemma}
\begin{proof}
  The space of $2$-chains is 
  \begin{equation}\label{}
    \op{S}^j (\sP) \otimes (\wedge^2 \lie{w}^{poly}_2) .
  \end{equation}
  We will use the bi-weight decomposition with respect to the Euler vector
fields $(z^1 \del_{1}, z^2 \del_\2)$ which act on each individual tensor factor.
The subcomplex of biweight $(0,0)$:
\begin{equation}\label{}
  \left(\op{S}^j (\sP) \otimes (\wedge^2 \lie{w}^{poly}_2) \right)^{(0,0)}
\end{equation}
is quasi-isomorphic to the full chain complex.
We can therefore assume a $2$-cycle is a sum pure tensors of the form 
\begin{equation}\label{eq:cycle1}
  \PP^{(-m^1,-m^2)} \otimes 
\mathbb{v}^{(l^{\1},l^{\2})}\wedge\mathbb{w}^{(m^{\1}-l^{\1},m^{\2}
-l^{\2})} 
\end{equation}
where $\PP^{(-m^1,-m^2)} \in \op{S}^j(\sP)$ and the superscripts indicate bi-weight.
The first step in the proof is to reduce the possible values of $l^1,l^2$.

Note that $m^1,m^2 \geq j$, $m^1+m^2 \geq 3j$, and $l^1 + l^2 \geq -1$.
We will parametrize as 
\[
m^{\1}=j+s+a,\quad  m^{\2}=2j-s+b,\quad s=0,...,j,\quad j\geq 2,\quad a,b\geq 0.
\]

Suppose that 
\[
l^{\1}=-1,\quad l^{\2}=0,
\]
then $\mathbb{v}^{(-1,0)} = \del_\1$ and
\[
\mathbb{w}^{(j+s+1+a,2j-s+b)} =(z^{\1})^{j+s+1+a}(z^{\2})^{2j-s+b+1}\del_\2\quad \text{or}\quad (z^{\1})
^{j+s+2+a}(z^{\2})^{2j-s+b}\del_\1.
\]
In either case, we have
\[
  (z^{\1})^{j+s+1+a}(z^{\2})^{2j-s+b+1}\del_\2= - \frac{1}{j+s+a} \left[(z^{\1})^{j+s+a}(z^{\2})^{2j-s+b+1}\del_\2,
  (z^{\1})^{\2} \del_\1 \right]
\]
or
\[
  (z^{\1})^{j+s+2+a}(z^{\2})^{2j-s+b}\del_\1=-\frac{1}{j+s+a-1} \left[(z^{\1})^{j+s+1+a}(z^{\2})^{2j-s+b}\del_\1,
  (z^{\1})^{2} \del_\1 \right]
\]
In each of these expressions we are making use of the assumptions on our bounds for indices. 
In the second expression, the right hand side is well-defined since $j \geq
2$ and $s=0,\ldots,j$, $a \geq 0$.
Thus, when $l^1=-1, l^2=0$ our term \eqref{eq:cycle1} is of the form 
\begin{equation}\label{}
  \PP^{(-m^1,-m^2)} \otimes (\del_\1 \wedge [X,Y])
\end{equation}
where $X,Y$ both lie in the subalgebra: 
\begin{equation}\label{}
  L^{poly}_1 \subset \lie{w}_2^{poly}
\end{equation}
of polynomial vector fields with vanishing $1$-jet.
By applying the Chevalley--Eilenberg differential to the $3$-cycle 
\begin{equation}\label{}
\PP^{(-m^\1,-m^\2)} \otimes \del_\1 \wedge X \wedge Y ,
\end{equation}
we observe that all $2$-cycles which involve a weight $(-1,0)$ vector field can be replaced, up to exact terms, by
vector fields of weight $(l^1,l^2)$ where $l^1 \geq 0$.
Similarly, we can assume without loss of generality that $l^1 \geq 0$.

Next, we will argue that terms with $l^1+l^2 = 0$ do not contribute to cohomology.
Suppose that (note that $l^1+l^2 = 0, l^1 \geq 0$ implie that $l^{\1}=1,l^{\2}=-1$)
\[
\mathbb{v}^{(l^{\1}=1,l^{\2}=-1)}=z^{\1}\del_\2\in \mathfrak{gl}_2.
\]
then 
\[
\mathbb{w}^{(m^{\1}-l^{\1},m^{\2}-l^{\2})}=(z^{\1})^{j+s+a-1}(z^{\2})^{2j-s+b+2}\del_\2
\quad \text{or}\quad (z^{\1})^{j+s+a}(z^{\2})^{2j-s+b+1}\del_\1.
\]

If $j+s+a>2$, observe
\[
(z^{\1})^{j+s+a-1}(z^{\2})^{2j-s+b+2}\del_\2=-\frac{1}{j+s+a-2}[(z^{\1})^{j+s+a-2}(z^{\2})^{2j-s+b+2}\del_\2,(z^{\1})^{2}\del_\1]
\]
if $j+s+a=2$, then $a=0, j=2, s=0, m^{\2}=2j-s+b=4+b$.
In this case:
\[
(z^{\1})(z^{\2})^{6+b}\del_\2=-\frac{1}{3+b}[(z^{\1})(z^{\2})^{5+b}\del_\2,(z^{\2})^{2}\del_\2],
\]
On the other hand, for $j+s+a>3$ we have
\[
 (z^{\1})^{j+s+a}(z^{\2})^{2j-s+b+1}\del_\1=-\frac{1}{j+s+a-3}[ (z^{\1})^{j+s+a-1}(z^{\2})^{2j-s+b+1}\del_\1,(z^{\1})^{2}\del_\1],
\]
 Finally, if $j+s+a\leq 3$, then $j=2, s\leq 1$, we have $m^{\2}=2j-s+b=4-s+b$
$$
 (z^{\1})^3(z^{\2})^{5-s+b}\del_\1=-\frac{1}{4-s+b}[ (z^{\1})^{3}(z^{\2})^{4-s+b}\del_\1,(z^{\2})^{2}\del_\2].
$$

We conclude that without loss of generality that a cycle in $Z_2^{Lie}(\lie{w}_2^{poly} ; \op{S}^j \sP )
$ is of the form
\begin{equation}\label{eq:cycletriv}
\sum \mathbb{P}^{(-m^{\1},-m^{\2})}_i\wedge \mathbb{v}^{(l^{\1},l^{\2})}_i\wedge\mathbb{w}^{(m^{\1}-l^{\1},m^{\2}-l^{\2})}
_i,\quad \mathbb{P}^{(-m^{\1},-m^{\2})}_i\in(\op{S}^j \sP),\quad \mathbb{v}^{(l^{\1},l^{\2})}_i\wedge\mathbb{w}^{(m^{\1}-l^{\1},m^{\2}-l^{\2})}_i\in \wedge^2L_1,
\end{equation}
where
\[
m^{\1}=j+s+a,\quad  m^{\2}=2j-s+b,\quad s=0,...,j,\quad j\geq 2,\quad a,b\geq 0.
\]
Our strategy is to now show that every such cycle is exact. Each term in the cycle above is labeled by integers $(m_1,m_2,\ell_1,\ell_2)$. Let $m^*$ be the maximum of $m^{\1}+m^{\2}$. We will argue that one can find homologically equivalent replacements for the expression with strictly smaller $m^*$.
The key to the argument is that our vector fields are assumed to lie in $L_1^{poly}$, So, if we can show that one, or both,
are commutators of other vector fields in $L_1^{poly}$, then the form of the Chevalley--Eilenberg differential will create a homotopy equivalence between the original cycle and one where the integers $(m^1,m^2)$ which appear in the sum are
strictly smaller.
We then repeat this until we find that the cycle is exact, which happens in some finite stage since no term has $m^1 + m^2 < 3j$.

We begin with the maximal value of $a+b$ (and hence of \(m^{\1}+m^{\2}\)) appearing in the cycle \eqref{eq:cycletriv}. Since \(m^{\1}+m^{\2}\) is assumed to be maximal, the 2-chain
\[
\sum \mathbb{v}^{(l^{\1},l^{\2})}_i \wedge \mathbb{w}^{(m^{\1}-l^{\1},m^{\2}-l^{\2})}_i \in \wedge^2 L_1^{poly}
\]
is a cycle. By \cite[Prop.~A.9]{FuksBook}, a basis of $H_1(L_1^{poly})$ (respectively, $H_2(L_1^{poly}$)) may be viewed as a minimal system of generators of $L_1^{poly}$ (respectively, a minimal system of defining relations). According to \cite[Appendix.~D]{FuksBook}, $L_1^{poly}$ has 6 weight-1 generators, 7 weight-2 relations, and 18 weight-3 relations. Hence $H_2(L_1^{poly})$ has no component of weight $\ge 4$. Since \(m^{\1}+m^{\2}\ge 3j\ge 6>4\), it follows that the above 2-chain is exact. Therefore, we can replace
\[
\sum \mathbb{P}^{(m^{\1},m^{\2})}_i \wedge \mathbb{v}^{(l^{\1},l^{\2})}_i \wedge \mathbb{w}^{(m^{\1}-l^{\1},m^{\2}-l^{\2})}_i
\]
by terms involving \(\mathbb{P}^{(m^{'\1},m^{'\2})}_i\) with \(m^{'\1}+m^{'\2}<m^{\1}+m^{\2}\). Repeating this procedure completes the proof.
\end{proof}

Our final calculation is of the remaining summand on the right hand side of \eqref{eq:decomp}.
Let $L_k \subset \lie{w}_d$ denote the ideal of vector fields on the formal $d$-disk whose $k$-jet vanishes.
Observe that there is a natural isomorphism $L_0 \slash L_1 = \lie{gl}_d$.
  Let $\sM$ be any tensor module of $\lie{w}_2$.
  That is, a module which is the coinduction of a tensor $\lie{gl}_d$-module $M$ along
  \begin{equation}\label{}
    \lie{gl}_d \leftarrow L_0 \subset \lie{w}_d .
  \end{equation}
  Using the Hocschild--Serre spectral sequence it is shown \cite{Losik} how to express the Lie algebra cohomology of
  $\lie{w}_d$ with coefficients in such tensor modules in terms of the $GL_d$-representation theory of $H^\bu_{Lie}(L_1)
  $.
  Note that this cohomology is generally very difficult to compute as it does not contain a reductive subalgebra (like
  for $\lie{w}_d$).
  The result is:
    \begin{equation}\label{}
    H^\bu(\lie{w}_d \, ; \, \sM) \cong H^\bu(\lie{gl}_d) \otimes \left[H^\bu(L_1) \otimes M\right]^{GL_d}
    \end{equation}
    We use this result in our last lemma.
\begin{lemma}
  One has 
  \begin{equation}
    \dim H^2_{Lie} (\lie{w}_2 \, ; \, \Omega^1 \otimes \Omega^2) = 2 .
  \end{equation}
\end{lemma}
\begin{proof}
  It suffices to show:
\begin{align*}
  \dim \left(H^q_{Lie}(L_1) \otimes V^* \otimes \wedge^2 V^*\right)^{GL_2} & = 0, \quad q=0,1 \\
  \dim \left(H^2_{Lie} (L_1) \otimes V^* \otimes \wedge^2 V^* \right)^{GL_2} & = 2 .
  \end{align*}
  
  We consider the $\lie{sl}_2$ representation type and scaling dimension. 
  Keeping the notation of the previous lemma, we have the decomposition 
  \begin{equation}\label{}
    V^* \otimes \wedge^2 V^* = S \otimes U^3   
  \end{equation} 
    Thus, scaling invariants will arise completely from $U^{-3}$ summands in $H^q_{Lie}(L_1)$.
    The decomposition of the continuous linear dual of $L_1$ 
  \begin{align}\label{eq:L1sl2}
    L_1^* & = \op{S}^{\geq 2} (S \otimes U^{-1}) \otimes \op{S} \otimes U \notag \\ & = \op{S} \otimes U^{-1} \oplus \op{S}^2 \otimes
    U^{-2} \oplus
    \bigoplus_{m \geq 3} \op{S}^m \otimes (U^{-m+2} \oplus U^{-m})
  \end{align}
  
  When $q=0$ there are obviously no invariants.
  When $q=1$ we need to consider the $U^3$ part of $L_1^*$ which is $\op{S}^3 \oplus \op{S}^5$.
  Therefore the $U^{-3}$ component of $H^1(L_1) \otimes \op{S}$ will not contain any $\lie{sl}_2$-invariant summands.

  Finally, we need to consider $H^2_{Lie}(L_1)$.
  There are the following components to $U^{-3}$ in~$\wedge^2 L_1^*$:
  \begin{itemize}
    \item[(1)] $\op{S} \otimes \op{S}^2$.
    \item[(2)] $\op{S} \otimes \op{S}^4$.
    \item[(3)] $\op{S}^2 \otimes \op{S}^3$.
    \item[(4)] $\op{S}^3 \otimes \op{S}^4$.
  \end{itemize}
  The multiplicity of $\op{S}$ in the space of $2$-cochains is therefore three: one from (1), one from (3), and one
  from (4).
  We can identify these cochains explicitly. 
  We work with Lie algebra homology, and chains, for convenience.
  
  The highest-weight vector in the $\op{S}$ summand in (1) is dual to the $2$-chain 
  \begin{equation}\label{}
    \alpha \define z^{[}\sfE\sfu \wedge z^\1 z^{]} \sfE\sfu = z^\1 \sfE\sfu \wedge z^\1 z^\2  \sfE\sfu - z^\2 \sfE\sfu
    \wedge (z^\1)^2
    \sfE\sfu . 
  \end{equation} 
  The highest-weight vector in the $\op{S}$ summand in (3) is dual to the $2$-chain: 
  \begin{multline}
    \beta \define
    3 (z^\1)^2 \del_\2 \wedge (z^\2)^2 \sfE \sfu + 2 ((z^\1)^2 \del_\1 - 2 z^\1 z^\2 \del_\2) \wedge z^\1 z^\2 \sfE \sfu - (2 z^\1
    z^\2 \del_\1 - (z^\2)^2 \del_\2) \wedge (z^\1)^2 \sfE \sfu .
  \end{multline}
  Finally, the highest-weight vector in the $\op{S}$ summand in (4) is dual to the $2$-chain:
  \begin{align}\label{}
    \gamma & 
    \define 3 (z^\1)^2 \del_\2 \otimes (3 z^\1 (z^\2)^2 \del_\1 - (z^\2)^3 \del_\2) \\ & + 3 ((z^\1)^2 \del_\1 - 2 z^\1
    z^\2 \del_\2) \wedge ((z^\1)^2 z^\2 \del_\1 - z^\1 (z^\2)^2 \del_\2)  \\
           &- (2 z^\1 z^\2 \del_\1 - (z^\2)^2 \del_\2) \wedge ((z^\1)^3 \del_\1 - 3 (z^\1)^2 z^\2 \del_\2) \\ 
           & - 4 (z^\2)^2 \del_\1 \otimes (z^\1)^3 \del_\1 .
\end{align}

Note that at the appropriate scaling dimension, which is $3$ at the level of chains, there are no $\op{S}$ summands in $\wedge^1 L^{poly}_1 = L^{poly}_1$.
Therefore, these cochains are automatically closed.
It is also instructive to check this by direct calculation.
Thus, there are at most three copies of $\sfS$ which survive in cohomology.

On the other hand, there is a unique copy of $\op{S}$ in $\wedge^3 L_1^{poly}$ of scaling dimension $3$.
It comes from $\op{S} \otimes U \otimes \wedge^2 (\op{S}^3 \otimes U)$.
Its highest-weight vector is 
  \begin{align}\label{}
    \eta & \define  
    z^\1 \sfE \sfu \wedge (z^\1)^2 \del_\2 \wedge (z^\2)^2 \wedge \del_\1 + \frac13 z^\1 \sfE \sfu \wedge ((z^\1)^2
    - 2 z^\1 z^\2 \del_\2) \wedge (2 z^\1 z^\2 \del_\1 - (z^\2)^2 \del_\2) . 
  \end{align}
  It is easy to see that $\d_{CE} \eta \ne 0$, so by Schur's lemma the result follows.
  (In fact, direct calculation shows $\d_{CE} \eta = \frac14 (-\beta + \gamma)$.
  So, nice nontrivial homology representatives are $[\alpha], [\beta + \gamma] \in H_2(L_1)$).
\end{proof}

We have finally completed the proof of theorem \ref{thm:2dgf}.


\subsection{Homology of differential operators}

We will need the next result in our proof of the local, universal form of the Grothendieck--Riemann--Roch theorem
contained in the next section.

We consider the Jouanolou dg algebra of (polynomial) differential operators on the punctured affine space $\cD^{poly}_2$ as
defined in \S \ref{s:weyl}.
We refer to this section for the notations used here.

The main tool we use to prove the following theorem is a similar spectral sequence as we used to compute the second Lie algebra (co)
homology of $\lie{witt}_2$.

\begin{theorem}\label{thm:diffops}
  The second hyperhomology $\mathbb{H}_2^{Lie}(\cD_2^{poly})$ is one-dimensional and generated by the following class
  \[
    [P\wedge z^{\1}\wedge z^{\2}].
  \]
\end{theorem}
\begin{proof}
  We use the same type of spectral sequence as we used in the case of $\lie{witt}_2$ in the proof of lemma \ref{lem:sswitt}.
  Namely, the spectral sequence whose first page takes the internal, or $\dbar$, cohomology. We will prove the following:

    \[
    H^{Lie}_1\left(\fD_2^{poly};\op{S}^j(\til \sP) \right)=0,\quad j\geq 0, 
  \]
      \[
    H^{Lie}_2\left(\fD_2^{poly};\til \sP \right)\simeq \C, 
  \]
      \[
    H^{Lie}_2\left(\fD_2^{poly};\op{S}^j(\til \sP) \right)=0, \quad j\geq 2.
  \]
Where $\til \sP = H^1_{\dbar}(\cD_2^{poly})$ which carries a natural $H^0_{\dbar}(\cD_2^{poly})=\fD_2^{poly}$-action. A typical element $\mathbb{P}^{-m^{\1},-m^{\2}}_{-p^{\1},-p^{\2}}\in \op{S}^j(\til \sP)$ can be written as
\[
{P}^{-m^{\1}_1,-m^{\2}_1}_{-p^{\1}_1,-p^{\2}_1}\cdots {P}^{-m^{\1}_j,-m^{\2}_j}_{-p^{\1}_j,-p^{\2}_j},\quad  {P}^{-m^{\1}_i,-m^{\2}_i}_{-p^{\1}_i,-p^{\2}_i}=\left((\del_\1)^{p^{\1}_i}(\del_\2)^{p^{\2}_i}P\right)\cdot (\del_\1)^{m^{\1}_i-1}(\del_\2)^{m^{\2}_i-1},
\]
\[
\sum^j_{i=1}m^{\s}_i=m^{\s},\quad \sum^j_{i=1}p^{\s}_i=p^{\s},\quad \s=1,2.
\]
  
  For the first claim, we only need the following
  \[
  (z^{\1})^{m^{\1}+p^{\1}+l^{\1}}  (z^{\2})^{m^{\2}+p^{\2}+l^{\2}}(\del_\1)^{l^{\1}}(\del_\2)^{l^{\2}}=\frac{1}{l^{\1}+1}[(z^{\1})^{m^{\1}+p^{\1}+l^{\1}}  (z^{\2})^{m^{\2}+p^{\2}+l^{\2}}(\del_\1)^{l^{\1}+1}(\del_\2)^{l^{\2}},z^{\1}].
  \]
  Then 
  $$
 \sum\mathbb{P}^{-m^{\1},-m^{\2}}_{-p^{\1},-p^{\2}}\wedge   (z^{\1})^{m^{\1}+p^{\1}+l^{\1}}  (z^{\2})^{m^{\2}+p^{\2}+l^{\2}}(\del_\1)^{l^{\1}}(\del_\2)^{l^{\2}}
  $$
  can be reduced to zero as $\op{S}^j(\til \sP)$ only contains elements with weight $\geq 2$.
  For the second claim, we first reduce the following
  $$
  \sum{P}^{-m^{\1},-m^{\2}}_{-p^{\1},-p^{\2}}\wedge   (z^{\1})^{l^{\1}}  (z^{\2})^{l^{\2}}(\del_\1)^{r^{\1}}(\del_\2)^{r^{\2}}\wedge (z^{\1})^{k^{\1}}  (z^{\2})^{k^{\2}}(\del_\1)^{-m^{\1}-p^{\1}+l^{\1}-r^{\1}+k^{\1}}(\del_\2)^{-m^{\2}-p^{\2}+l^{\2}-r^{\2}+k^{\2}}
  $$
  to
    $$
  \sum{P}^{-m^{\1},-m^{\2}}_{0,0}\wedge   (z^{\1})^{l^{\1}}  (z^{\2})^{l^{\2}}(\del_\1)^{r^{\1}}(\del_\2)^{r^{\2}}\wedge (z^{\1})^{k^{\1}}  (z^{\2})^{k^{\2}}(\del_\1)^{-m^{\1}+l^{\1}-r^{\1}+k^{\1}}(\del_\2)^{-m^{\2}+l^{\2}-r^{\2}+k^{\2}}.
  $$
  Then use
  $$
  P(\del_\1)^{m^{\1}-1}(\del_\2)^{m^{\2}-1}=\frac{1}{m^{\1}}[  P,z^{\1}(\del_\1)^{m^{\1}}(\del_\2)^{m^{\2}-1}]
  $$
  to reduce to
    $$
P\wedge  \sum (z^{\1})^{l^{\1}}  (z^{\2})^{l^{\2}}(\del_\1)^{r^{\1}}(\del_\2)^{r^{\2}}\wedge (z^{\1})^{k^{\1}}  (z^{\2})^{k^{\2}}(\del_\1)^{-1+l^{\1}-r^{\1}+k^{\1}}(\del_\2)^{-1+l^{\2}-r^{\2}+k^{\2}}.
  $$
If that $l^{\1}\geq r^{\1}+2$, or $l^{\1}= r^{\1}+1, l^{\2}\geq r^{\2}$ we can use
$$
(z^{\1})^{l^{\1}}  (z^{\2})^{l^{\2}}(\del_\1)^{r^{\1}}(\del_\2)^{r^{\2}}=[z^{\1},(z^{\1})^{l^{\1}}  (z^{\2})^{l^{\2}}(\del_\1)^{r^{\1}+1}(\del_\2)^{r^{\2}}]
$$
to reduce to
$$
P\wedge z^{\1}\wedge \sum  (z^{\1})^{k^{\1}}  (z^{\2})^{k^{\2}}(\del_\1)^{k^{\1}}(\del_\2)^{k^{\2}-1}.
$$
Suppose that $l^{\1}= r^{\1}+1, l^{\2}<r^{\2}$, then we reduce to
$$
P\wedge z^{\2}\wedge \sum  (z^{\1})^{k^{\1}}  (z^{\2})^{k^{\2}}(\del_\1)^{k^{\1}-1}(\del_\2)^{k^{\2}}.
$$
which again reduce to $
P\wedge z^{\1}\wedge \sum  (z^{\1})^{k^{\1}}  (z^{\2})^{k^{\2}}(\del_\1)^{k^{\1}}(\del_\2)^{k^{\2}-1}.
$ The cycle condition implies that
$$
P\wedge z^{\1}\wedge \sum   (z^{\2})^{k^{\2}}(\del_\2)^{k^{\2}-1}.
$$
Now use
$$
 (z^{\2})^{k^{\2}}(\del_\2)^{k^{\2}-1}=-\frac{1}{k^{\2}}[(z^{\2}), (z^{\2})^{k^{\2}}(\del_\2)^{k^{\2}}]
$$
we finally arrive at 
$$
P\wedge z^{\1}\wedge z^{\2}.
$$
We turn to the third claim
      \[
    H^{Lie}_2\left(\mathcal{D}_2;\op{S}^j(\til \sP) \right)=0, \quad j\geq 2.
  \]
  Consider a cycle in the form
  the following
  $$
  \sum\mathbb{P}^{-m^{\1},-m^{\2}}_{-p^{\1},-p^{\2}}\wedge   (z^{\1})^{l^{\1}}  (z^{\2})^{l^{\2}}(\del_\1)^{r^{\1}}(\del_\2)^{r^{\2}}\wedge (z^{\1})^{k^{\1}}  (z^{\2})^{k^{\2}}(\del_\1)^{-m^{\1}-p^{\1}+l^{\1}-r^{\1}+k^{\1}}(\del_\2)^{-m^{\2}-p^{\2}+l^{\2}-r^{\2}+k^{\2}}
  $$
  We can use the same method to reduce to
  $$
  \sum\mathbb{P}^{-m^{\1},-m^{\2}}_{-p^{\1},-p^{\2}}\wedge   z^{\1}\wedge (z^{\1})^{k^{\1}}  (z^{\2})^{k^{\2}}(\del_\1)^{-m^{\1}-p^{\1}+1+k^{\1}}(\del_\2)^{-m^{\2}-p^{\2}+k^{\2}}.
  $$
  Consider the $\mathbb{P}^{-m^{\1},-m^{\2}}_{-p^{\1},-p^{\2}}$, for which $m^{\1}+m^{\2}+p^{\1}+p^{\2}$ is maximum, we get the following from the cycle condition
    $$
  \sum\mathbb{P}^{-m^{\1},-m^{\2}}_{-p^{\1},-p^{\2}}\wedge   z^{\1}\wedge (z^{\1})^{m^{\1}+p^{\1}-1}  (z^{\2})^{k^{\2}}(\del_\2)^{-m^{\2}-p^{\2}+k^{\2}}.
  $$
  Note that $m^{\2}\geq 2$, this term can be reduce to lower $\mathbb{P}^{-m^{\1},-m^{\2}}_{-p^{\1},-p^{\2}}$. Repeating this, we get the claim.
\end{proof}

The Jouanolou model $\cJ_2$ has the structure of a (dg) Tate vector space as shown in \cite{FHK}. 
It follows, see \cite{FHK}, that the second Lie algebra homology
of the Lie algebra $\mathrm{End}(\cJ_2)$ of continuous endomorphisms of $\cJ_2$ is one-dimensional. 
The dg Lie algebra~$\mathcal{D}^{poly}_{2}$ naturally maps to $\mathrm{End}(\cJ_2)$ and we have the following corollary.
\begin{corollary}
    The natural map 
    \[
      \mathbb{H}_2^{\mathrm{Lie}}(\mathcal{D}^{poly}_{2})\rightarrow \mathbb{H}_2^{\mathrm{Lie}}\big(\mathrm{End}
      (\cJ_2)\big)\simeq \C 
    \]
    is an isomorphism.
\end{corollary}

\section{Local Grothendieck-Riemann-Roch theorem}\label{s:2dGRR}

In this final section we formulate and prove a local, universal form of the Grothendieck-Riemann-Roch theorem for
formal families of complex surfaces.
It is a natural generalization, to complex dimension two, of the relationship of the Mumford determinant line bundle defined on the
moduli space of Riemann surfaces and the central charge. 
There is a natural extension to arbitrary dimensions which we will return to in future work.

The set up is as follows.
First, note that $\cJ_d$ is a representation for the dg lie algebra $\lie{witt}_d$
\[
 \sfL \colon \lie{witt}_d \xto{\simeq} \op{Der}( \cJ_d) \to \op{End} \cJ_d .
\]
Pulling back along $\sfL$ is a map in Lie algebra hyper Lie algebra cohomology 
\begin{equation}\label{}
  \sfL^* \colon \HH^2_{Lie}(\op{End} \cJ_d) \to \HH^2_{Lie} (\lie{witt}_d) . 
\end{equation}

More generally, let $\cV$ be any tensor $\lie{witt}_d$-module; that is, the Jouanolou model of some tensor bundle $V$ on the
punctured formal disk $\cV \simeq \RR \Gamma(\mathring{D}^d , V)$. 
Then, we obtain a homomorphism (essentially Lie derivative): 
\begin{equation}\label{}
  \sfL_{V} \colon \lie{witt}_d \to \op{End}(\cV)  
\end{equation}
and hence a map in cohomology 
\begin{equation}\label{}
  \sfL_V^* \colon \HH^2_{Lie}(\op{End} \cV) \to \HH^2_{Lie}(\lie{witt}_d) .
\end{equation}

By \cite[theorem 5.4.16]{FHK} the dg vector space $\op{End} \cV$ is a dg Tate vector space, and therefore its second Lie algebra
cohomology is one-dimensional, with generator we denote 
\begin{equation}\label{}
  \mathfrak{t} \in \HH^2_{Lie}(\op{End}\cV) \simeq \C.
\end{equation}
It is dual to the homology class
\begin{align}\label{eq:universalV}
  \left[z^\1 \wedge \cdots \wedge z^{\ddd} \wedge P \right]  &  \in 
  \HH_2^{Lie}( \op{End} \cV) \\ 
  \<\mathfrak{t}, [z^\1 \wedge \cdots \wedge z^{\ddd} \wedge P] \> & = 1 .
\end{align}
Here, we use that $z^\1, \ldots, z^{\ddd}, P \in \cJ_d$ acts purely through the $\cJ_d$-module structure.
With this setup, we have the following local, universal version of the Grothendieck-Riemann-Roch theorem which we state as a conjecture.

\begin{conjecture}[see also \cite{KapranovNotes}] \label{conj}
  The composition
\begin{equation}\label{}
  \HH_2^{Lie}(\lie{witt}_d) \xto{(\sfL_V)_*} \HH_2^{Lie}( \op{End} \cV) \xto{\tau} \C ,
\end{equation}
is the following universal class:
  \begin{equation}\label{}
    \op{Td} \op{ch}(V)|_{2d+2} \in \C[\op{ch}_1,\ldots,\op{ch}_d]_{2d+2}
  \end{equation}
  In other words, at the level of cohomology, one has 
  \begin{equation}\label{}
    \sfL^* [z^\1 \wedge \cdots \wedge z^{\ddd} \wedge P] = \left[\op{Td} \op{ch}|_{2d+2}\right] \in \HH^2_{Lie}(\lie{witt}_d) .
  \end{equation}
\end{conjecture}

In this section, we only consider the case $d=2$ and we will prove the following restatement of theorem \ref{thm:GRRglobal}.

\begin{theorem}\label{thm:grr2}
  The conjecture is true in the case $d=2$ and $\cV = \cJ_2$ (trivial tensor bundle).
\end{theorem}

\subsection{Differential operator formulation} 

We sketch our method of proof of theorem \ref{thm:grr2} before getting into details.
The main idea is to involve the (dg) algebra of differential operators $\cD_d$ on punctured affine space $\AA^d - \{0\}
$, see section \ref{s:weyl}.
By construction, $\sfL \colon \lie{witt}_d \to \op{End} \cJ_d$ factors through differential operators
\begin{equation}\label{}
  \lie{witt}_d \xto{i} \cD_d \xto{\til \sfL} \op{End} \cJ_d 
\end{equation}
where the first map simply views vector fields acting by first-order differential operators and the second map is the natural inclusion.
From here, we will use theorem \ref{thm:diffops} together with the existence of an algebraic trace in the spirit of
\cite{FFS}.
Indeed, in the next section we will define the ``residue trace" on our Jouanolou model for differential operators as a degree two class 
\begin{equation}\label{eq:trD}
  \op{ResTr}_{\cD}^{S^1} \in \HH^2_{Lie}(\cD_{d}) . 
\end{equation}
Consider the commutative diagram 
\begin{equation}\label{eq:diagram}
  \begin{tikzcd}
    \HH_2^{Lie}(\lie{witt}_d) \ar[rr,"\sfL_*"] \ar[dr,"i_*"] & & \HH_2^{Lie}(\op{End} \cJ_d) \ar[r,"\tau","\cong"'] & \C \\
                                                                     & \HH_2^{Lie}(\cD_d) \ar[rr,"\op{ResTr}^{S^1}_{\cD}",
    "\cong"'] \ar[ur,"\til \sfL_*"] & &
                                                                 \C 
                                                                 \end{tikzcd}
                                                               \end{equation}
By definition $\tau([z^1 \wedge z^2 \wedge \cdots \wedge z^d \wedge P]) = 1$
where $\tau$ is the universal Tate class of \cite{FHK}.
We will also show that $\op{ResTr}_{\cD}^{S^1}([z^1 \wedge \cdots \wedge z^d \wedge P]) = 1$.

\begin{theorem}\label{thm:grrD}
The composition 
\begin{equation}\label{}
  \HH_2^{Lie}(\lie{witt}_d) \xto{i_*} \HH_2^{Lie}(\cD_d) \xto{\op{ResTr}_{\cD}^{S^1}} \C .
\end{equation}
is $\op{Td}|_{2d+2}$.
\end{theorem}

If, in addition, we know that $\HH_2^{Lie}(\cD_d) \cong \C$ and that $\til \sfL_*$ is an isomorphism, then from the commutative
diagram above we see that this theorem
immediately implies the $\cV=\cJ_d$ version of conjecture \ref{conj}.
This result is well-known when $d=1$, see \cite[Theorem.~2]{li1998central}.
In the present paper, we have only extended this result about Lie algebra cohomology of $\cD_d$ to dimension $d=2$, see theorem \ref{thm:diffops}. 
In the next section, we prove theorem \ref{thm:grrD} for general dimension $d$ and hence theorem \ref{thm:grr2} follows. 
We will return to the computation of $\HH_2^{Lie}(\cD_d)$, and hence the proof of the general form of conjecture
\ref{conj} in future work.

\subsection{Algebraic residue trace}

As we have already mentioned, this class \eqref{eq:trD} is inspired by the formal algebraic trace map of \cite{FFS}.
We introduce the dg Weyl algebra $\cW_d$ associated to the punctured formal disk, see section \ref{s:weyl} for definitions and
notations.
There is a dg algebra isomorphism from $\mathcal{D}_{d}$ to $\mathcal{W}_{d}$ given by the symbol map
\[
\sigma\left(f(z^{\1},...,z^{\dd})F(\partial_{z^{\1}},...,\partial_{z^{\dd}})\right) \define e^{-\frac{1}{2}\sum^d\limits_{\s=1}\partial_{p^{\s}}\partial_{q^{\s}}}(f(q^{\1},...,q^{\dd})F(p^{\1},...,p^{\dd})).
\]
In particular
\[
\sigma\left(T\right)=T-\frac{1}{2}\mathrm{div}(T)
\]
where $T\in \lie{witt}_d \subset \cD_d$ and $\op{div}(\sum_i f^i \del_i) = \sum_i \del_i f^i \in \cJ_d$.

The residue trace is most fundamentally defined at the level of cyclic cohomology.
It is a \textit{non-commutative} version of the cyclic class $\rho \in \HH_\lambda^1(\cJ_d)$ of \cite{FHK}. We follow the presentation in \cite[3]{GLX}, which gives a convenient setting to generalize \cite{FFS} to our case.

\begin{definition}
    Define the residue trace cochain
    \[
      \op{ResTr}^{S^1}_\lambda \in  \overline{C}_{\lambda}^{\bu}(\mathcal{W}_d)^1
  \]
   of total degree $+1$ by the formula
\begin{multline}
  \mathrm{ResTr}_\lambda^{S^1}\left(O_0\otimes \cdots\otimes O_k\right) =
  \mathrm{Res}_{q} \left({\int_{S^1_{\mathrm{cyc}}[k+1]}\mathbf{Mult}\left(e^{\Pi+D}(O_0\cdot d\theta_0\otimes O_1\otimes\cdots\otimes O_k)\right)}|_{p=0}\right),
\end{multline}
    where:
    \begin{itemize}
    \item[(i)] the operators $\Pi$ and $D$ are as follows:
      \[
\Pi\define \frac{1}{2}\sum_{i<j}\pi^*_{ij}\SP\cdot \Pi_{ij},\quad \Pi_{ij}:=\sum^d_{\s=1}(\partial_{q^{\s}})_i\otimes(\partial_{p^{\s}})_j-(\partial_{p^{\s}})_i\otimes(\partial_{q^{\s}})_j,
\]
\[
D(O_0\cdot d\theta_0\otimes O_1\otimes\cdots\otimes O_k)=\sum^{k}_{i=1}O_0\cdot d\theta_0\otimes \cdots\otimes D_i(O_i)\otimes\cdots\otimes O_k,
\]
    \[
    D_i(O_i)\define \sum^d_{\s=1}{\partial_{q^{\s}}}(O_i)dq^{\s}\cdot d\theta_i+\sum^d_{\s=1}{\partial_{p^{\s}}}(O_i)dp^{\s}\cdot d\theta_i.
  \]
      \item[(ii)] Let $\mathrm{Cyc}_{S^1}[k+1]$ be the configuration space of $k+1$ anti-clockwise cyclic ordered points on the circle $S^1$. We have a natural identification
      \[
      \mathrm{Cyc}_{S^1}[k+1]\simeq S^1\times \mathring{\Delta}_{k}
      \]
      \[
      (\theta_0,\dots,\theta_k)\rightarrow \{p_0\}\times (\theta_{0,1},\theta_{1,2},\dots,\theta_{k,0}).
      \]
      Here we denote $\Delta_k=\{(\theta_{0,1},\theta_{1,2},\dots,\theta_{k-1,k},\theta_{k,0})\in [0,1]^{k+1}| \theta_{0,1}+\theta_{1,2}+\dots+\theta_{k,0}=1\}$ viewed as a manifold with corners and $\mathring{\Delta}_{k}$ its interior. This allows us to compactify $\mathrm{Cyc}_{S^1}[k+1]$ by $S^1\times\Delta_{k}$, which will be denoted by $S^1_{\mathrm{cyc}}[k+1]$.  
      \item[(iii)] For $0\leq i<j\leq k$, we have the fogetful map
      \[
\pi_{ij}:      S^1_{\mathrm{cyc}}[k+1]\rightarrow       S^1_{\mathrm{cyc}}[2]
      \]
      \[
      (\theta_0,\dots,\theta_k)\rightarrow (\theta_i,\theta_j).
      \]
    Recall  that $S^1_{\mathrm{cyc}}[2]\simeq S^1\times \Delta_1=S^1\times [0,1]$. The function $\SP$ on $S^1_{\mathrm{cyc}}[2]$ is define to be $u-\frac{1}{2}$ where $u\in \Delta_1=[0,1]$.
    
      \item[(iv)] The notation $\mathbf{Mult}\left(-\right)|_{p=0}$ means that we first multiply together the tensor components and then set $p^{\s}=dp^{\s}=0$.
      \item[(v)] $\op{Res}_q$ is the same Jouanolou residue introduced in section \ref{s:res} except we use the
        variable $q$ instead of $z$.
      \end{itemize}
\end{definition}
From the definition, $\mathrm{ResTr}^{S^1}$ is zero except $k=d$.

Since the residue $\mathrm{Res}$ annihilates holomorphic total derivatives we see immediately that $\op{Tr}^{S^1}$ is
well-defined: 
\[
\op{ResTr}_\lambda^{S^1}\left(O_0\otimes \cdots\otimes O_k\right)= (-1)^{(|O_0|-1)\cdot(\sum^k\limits_{i=1}|O_i|-1)}\cdot
\mathrm{ResTr}_\lambda^{S^1}\left(O_1\otimes \cdots\otimes O_k\otimes O_0\right)
\]

Next, we show that $\op{ResTr}_\lambda^{S^1}$ is a cocycle.
\begin{lemma}
    $\op{ResTr}_\lambda^{S^1}$ is closed for both Jouanolou and Hochschild differentials.
\end{lemma}
\begin{proof}
  Since the residue map $\mathrm{Res}$ vanishes on Jouanolou exact forms and the operator $(\partial_{q^{\s}})_i
  \otimes (\partial_{p^{\s}})_j$ commutes with the Jouanolou differential, it follows that the trace $\mathrm{ResTr}^{S^1}$ also vanishes on Jouanolou boundaries. 

Employing integration by parts, we obtain
\begin{align*}
&\mathrm{ResTr}_\lambda^{S^1}\bigl(b(O_0\otimes \cdots\otimes O_{d+1})\bigr)\\
&=\mathrm{Res}_{q}\left(\left.{\int_{S^1_{\mathrm{cyc}}[d+2]}d^{S^1}_{\mathrm{DR}}\Bigl(\mathbf{Mult}\Bigl(e^{\Pi}\Bigl(\sum^{d+1}_{i=1}O_0d\theta_0\otimes D(O_1)\otimes\cdots\otimes O_i\otimes\cdots\otimes D(O_{d+1})\Bigr)\Bigr)\Bigr)}\right|_{p=0}\right)\\
&=\mathrm{Res}_{q} \left({\int_{S^1_{\mathrm{cyc}}[k+2]}d^{S^1}_{\mathrm{DR}}\mathbf{Mult}\left(e^{\Pi+D}(O_0\cdot d\theta_0\otimes O_1\otimes\cdots\otimes O_{d+1}\right)}|_{p=0}\right).
\end{align*}
We introduce the following operator
\[
\mathbf{\triangle} \define \sum_{\s=1}^d \mathscr{L}_{\partial_{q^{\s}}}\iota_{\partial_{p^{\s}}}-\sum_{\s=1}^d
\mathscr{L}_{\partial_{p^{\s}}}\iota_{\partial_{q^{\s}}}
\]
where $\sL$ denotes the Lie derivative.
By the same argument in \cite[lemma 3.6]{GLX}, we have
$$
[d^{S^1}_{\mathrm{DR}}-\mathbf{\triangle},\Pi+D]=0,
$$
\[
(d^{S^1}_{\mathrm{DR}}-\mathbf{\triangle})\left(O_0\cdot d\theta_0\otimes O_1\otimes\cdots\otimes O_{d+1}\right)=0.
\]
Thus,
\begin{align*}
    &\mathrm{Res}_{q} \left({\int_{S^1_{\mathrm{cyc}}[k+2]}d^{S^1}_{\mathrm{DR}}\mathbf{Mult}\left(e^{\Pi+D}(O_0\cdot d\theta_0\otimes O_1\otimes\cdots\otimes O_{d+1}\right)}|_{p=0}\right)\\
    &=\mathrm{Res}_{q} \left({\int_{S^1_{\mathrm{cyc}}[k+2]}\mathbf{\triangle}\circ \mathbf{Mult}\left(e^{\Pi+D}(O_0\cdot d\theta_0\otimes O_1\otimes\cdots\otimes O_{d+1}\right)}|_{p=0}\right)\\
    &=0,
\end{align*}
here we use that fact that $\mathrm{Res}_{q}$ annihilates holomorphic total derivatives and at most $d$ copy of $dq^{\s}$.
\end{proof}

Define the cocycle 
\begin{equation}\label{}
  \op{ResTr}^{S^1} \in C_{Lie}^\bu(\cW_d)^2
\end{equation}
of total degree two to be the image of $\op{ResTr}_\lambda^{S^1}$ under  
\begin{equation}\label{}
  C_\lambda^\bu(\cW_d )[1] \to C_{Lie}^\bu(\cW_d)
\end{equation}
And let $\op{ResTr}_{\cD}^{S^1}$ denote the further twist of this map by the symbol:
\begin{equation}\label{}
  \op{ResTr}^{S^1}_{\cD} \define \sigma^* \op{ResTr}^{S^1}. 
\end{equation}

\subsection{One-loop calculation}

Recall the logarithmic Todd polynomial is
\[
\log(\mathrm{Todd}(y))=\frac{y}{2}-\sum_{m\geq 1}\frac{B_{2m}}{2m}y^m
\]
where $B_{2m}$ is $2m$th Bernoulli number.
We proceed to the proof of theorem \ref{thm:grrD}.
\begin{theorem*}[\ref{thm:grrD}]
  One has 
  \begin{equation}\label{}
    \op{ResTr}_{\cD}^{S^1}|_{\lie{witt}_d} = \op{Td}|_{2d+2}
  \end{equation}
\end{theorem*}

\begin{proof}
  The set up is similar to \cite{FFS} and \cite{GLX}, the major difference being that we twist by the symbol map (the
  secondary difference is that we
  use the residue instead of formal geometry).
  The key idea is that we can express the exponential in the style of Wick expansion which leads to a Feynman
  diagrammatic expansion for the residue trace.
  The next observation is that since we consider vector fields, as opposed to arbitrary order differential operators,
  only one-loop graphs appear in this expansion.

\begin{figure}[htbp]
    \centering
\begin{tikzpicture}[
    x=0.75pt,
    y=0.75pt,
    yscale=-1,
    xscale=1,
    line cap=round,
    line join=round,
    every node/.style={inner sep=0.75pt}
]

\tikzset{every picture/.style={line width=0.9pt}}

\definecolor{ptgreen}{RGB}{87,166,96}
\definecolor{ptpurple}{RGB}{137,92,191}
\definecolor{pathred}{RGB}{184,52,71}
\definecolor{pathorange}{RGB}{214,148,53}
\definecolor{framegray}{RGB}{55,61,68}
\definecolor{ellipsefill}{RGB}{250,250,251}
\definecolor{labelgray}{RGB}{38,43,48}

\tikzset{
  orbit/.style={draw=framegray, line width=1.15pt},
  redwiggle/.style={draw=pathred, line width=1.02pt, dash pattern={on 0.85pt off 0.85pt}, opacity=0.95},
  orangewiggle/.style={draw=pathorange, line width=1.02pt, dash pattern={on 0.85pt off 0.85pt}, opacity=0.95},
  greenpt/.style={draw=white, line width=0.95pt, fill=ptgreen},
  purplept/.style={draw=white, line width=0.95pt, fill=ptpurple},
  lab/.style={font=\scriptsize, text=labelgray, fill=white, fill opacity=0.96, text opacity=1, rounded corners=1.2pt, inner xsep=1.7pt, inner ysep=0.9pt},
  annobox/.style={draw=black!10, fill=white, rounded corners=2pt, inner xsep=7pt, inner ysep=4.5pt, text=labelgray, align=left}
}

\path[fill=ellipsefill]   (161.6,153.39) .. controls (145.57,124.52) and (156.06,88.17) .. (185.03,72.19) .. controls (214.01,56.22) and (250.5,66.67) .. (266.53,95.54) .. controls (282.57,124.41) and (272.08,160.76) .. (243.1,176.74) .. controls (214.13,192.71) and (177.64,182.26) .. (161.6,153.39) -- cycle ;
\draw[orbit]   (161.6,153.39) .. controls (145.57,124.52) and (156.06,88.17) .. (185.03,72.19) .. controls (214.01,56.22) and (250.5,66.67) .. (266.53,95.54) .. controls (282.57,124.41) and (272.08,160.76) .. (243.1,176.74) .. controls (214.13,192.71) and (177.64,182.26) .. (161.6,153.39) -- cycle ;

\filldraw[greenpt]   (177.35,77.98) circle [radius=3.75pt];
\filldraw[greenpt]   (274.17,117.7) circle [radius=3.75pt];
\filldraw[purplept] (203.8,65.82)  circle [radius=3.75pt];
\filldraw[purplept] (248.41,75.69) circle [radius=3.75pt];
\filldraw[purplept] (269.21,147.3) circle [radius=3.75pt];
\filldraw[purplept] (186.44,178.17) circle [radius=3.75pt];

\draw[redwiggle]  (177.35,77.98) .. controls (179.32,79.33) and (179.7,81.03) .. (178.51,83.08) .. controls (177.47,85.27) and (178.04,86.87) .. (180.22,87.9) .. controls (182.41,88.71) and (183.15,90.22) .. (182.43,92.44) .. controls (181.87,94.77) and (182.76,96.19) .. (185.09,96.7) .. controls (187.4,97.03) and (188.43,98.35) .. (188.18,100.67) .. controls (188.09,103.07) and (189.24,104.3) .. (191.64,104.37) .. controls (193.99,104.29) and (195.26,105.42) .. (195.43,107.77) .. controls (195.41,109.89) and (196.6,110.81) .. (199,110.52) .. controls (201.34,110.11) and (202.78,111.06) .. (203.31,113.38) .. controls (203.98,115.71) and (205.49,116.56) .. (207.84,115.95) .. controls (210.12,115.23) and (211.49,115.9) .. (211.95,117.95) .. controls (212.91,120.19) and (214.52,120.86) .. (216.78,119.97) .. controls (218.96,118.99) and (220.6,119.56) .. (221.7,121.69) .. controls (222.91,123.79) and (224.36,124.21) .. (226.06,122.94) .. controls (228.1,121.73) and (229.77,122.11) .. (231.06,124.09) .. controls (232.45,126.04) and (234.11,126.33) .. (236.04,124.94) .. controls (237.87,123.49) and (239.51,123.67) .. (240.96,125.48) .. controls (242.49,127.25) and (244.1,127.32) .. (245.77,125.71) .. controls (247.69,124.04) and (249.44,124) .. (251.02,125.6) .. controls (252.69,127.15) and (254.35,126.98) .. (256.02,125.1) .. controls (257.47,123.16) and (259.04,122.86) .. (260.73,124.19) .. controls (262.86,125.34) and (264.47,124.85) .. (265.54,122.72) .. controls (266.32,120.59) and (267.88,119.85) .. (270.23,120.52) .. controls (272.44,121.08) and (273.76,120.14) .. (274.17,117.7) -- (274.17,117.7) ;

\draw[redwiggle]  (177.35,77.98) .. controls (179.44,76.32) and (181.3,76.35) .. (182.92,78.07) .. controls (184.46,79.82) and (186.1,79.92) .. (187.85,78.37) .. controls (189.66,76.86) and (191.29,77.03) .. (192.75,78.88) .. controls (194.14,80.74) and (195.75,80.97) .. (197.6,79.57) .. controls (199.51,78.2) and (201.11,78.49) .. (202.41,80.44) .. controls (203.64,82.4) and (205.23,82.75) .. (207.16,81.48) .. controls (209.13,80.25) and (210.89,80.7) .. (212.42,82.83) .. controls (213.49,84.87) and (215.02,85.32) .. (217.01,84.19) .. controls (219.04,83.09) and (220.55,83.59) .. (221.52,85.68) .. controls (222.79,87.9) and (224.44,88.5) .. (226.48,87.49) .. controls (228.54,86.52) and (230.15,87.17) .. (231.3,89.45) .. controls (232.03,91.58) and (233.41,92.19) .. (235.46,91.29) .. controls (237.88,90.58) and (239.39,91.3) .. (240,93.45) .. controls (240.87,95.76) and (242.32,96.52) .. (244.37,95.71) .. controls (246.75,95.12) and (248.3,95.98) .. (249.01,98.31) .. controls (249.63,100.62) and (251.09,101.5) .. (253.4,100.97) .. controls (255.71,100.48) and (257.09,101.38) .. (257.54,103.68) .. controls (257.88,105.95) and (259.16,106.87) .. (261.39,106.42) .. controls (263.86,106.2) and (265.17,107.2) .. (265.3,109.43) .. controls (265.53,111.79) and (266.8,112.88) .. (269.12,112.7) .. controls (271.61,112.75) and (272.81,113.9) .. (272.72,116.15) -- (274.17,117.7) ;

\draw[orangewiggle]  (248.41,75.69) .. controls (249.52,73.03) and (251.15,72.04) .. (253.3,72.73) .. controls (255.35,73.54) and (256.65,72.88) .. (257.2,70.76) .. controls (258.52,68.38) and (260.13,67.72) .. (262.04,68.77) .. controls (264.44,69.73) and (266.04,69.23) .. (266.83,67.28) .. controls (268.4,65.21) and (269.97,64.87) .. (271.56,66.27) .. controls (273.61,67.66) and (275.46,67.45) .. (277.12,65.64) .. controls (278.31,63.93) and (279.82,63.91) .. (281.65,65.59) .. controls (283.28,67.36) and (285.03,67.52) .. (286.92,66.06) .. controls (288.95,64.72) and (290.65,65.06) .. (292,67.07) .. controls (293.15,69.12) and (294.5,69.53) .. (296.06,68.28) .. controls (298.26,67.33) and (299.81,67.95) .. (300.71,70.16) .. controls (301.42,72.37) and (303.11,73.28) .. (305.78,72.89) .. controls (308.09,72.4) and (309.43,73.31) .. (309.8,75.63) .. controls (309.99,77.91) and (311.21,78.93) .. (313.48,78.7) .. controls (315.8,78.63) and (316.9,79.76) .. (316.79,82.07) .. controls (316.5,84.3) and (317.63,85.72) .. (320.17,86.31) .. controls (322.42,86.68) and (323.23,87.97) .. (322.6,90.16) .. controls (322.01,92.65) and (322.77,94.21) .. (324.87,94.86) .. controls (326.96,95.69) and (327.5,97.31) .. (326.48,99.71) .. controls (325.17,101.44) and (325.47,103.09) .. (327.39,104.66) .. controls (329.17,106.01) and (329.23,107.66) .. (327.56,109.63) .. controls (325.79,110.85) and (325.59,112.49) .. (326.94,114.56) .. controls (328.19,116.41) and (327.71,118.01) .. (325.48,119.38) .. controls (323.39,120) and (322.61,121.55) .. (323.14,124.04) .. controls (323.64,126.29) and (322.73,127.56) .. (320.41,127.84) .. controls (318.12,127.91) and (316.97,129.11) .. (316.97,131.44) .. controls (316.68,133.87) and (315.29,134.99) .. (312.8,134.78) .. controls (310.85,134.07) and (309.5,134.93) .. (308.74,137.34) .. controls (308.33,139.46) and (306.79,140.24) .. (304.14,139.67) .. controls (302.17,138.7) and (300.81,139.26) .. (300.06,141.35) .. controls (298.37,143.7) and (296.49,144.32) .. (294.44,143.21) .. controls (292.41,142.04) and (290.77,142.46) .. (289.52,144.49) .. controls (288.11,146.5) and (286.8,146.77) .. (285.57,145.31) .. controls (283.47,143.99) and (281.61,144.3) .. (279.98,146.23) .. controls (278.2,148.14) and (276.71,148.32) .. (275.51,146.77) .. controls (273.31,145.3) and (271.75,145.44) .. (270.82,147.19) -- (269.21,147.3) ;

\draw[orangewiggle]  (248.41,75.69) .. controls (246.15,73.12) and (245.66,70.94) .. (246.94,69.14) .. controls (248.18,67.32) and (247.81,65.96) .. (245.82,65.06) .. controls (243.83,64.26) and (243.21,62.36) .. (243.94,59.35) .. controls (245.04,57.55) and (244.58,56.37) .. (242.56,55.82) .. controls (240.55,55.37) and (239.8,53.73) .. (240.31,50.91) .. controls (241.24,49.1) and (240.42,47.61) .. (237.84,46.46) .. controls (235.86,46.45) and (234.97,45.12) .. (235.18,42.47) .. controls (235.87,40.62) and (234.92,39.44) .. (232.33,38.91) .. controls (229.82,38.67) and (228.46,37.31) .. (228.27,34.84) .. controls (228.62,32.95) and (227.54,32.09) .. (225.03,32.26) .. controls (222.66,32.69) and (221.15,31.75) .. (220.5,29.45) .. controls (219.61,27.16) and (218.03,26.46) .. (215.75,27.33) .. controls (213.64,28.38) and (211.99,27.89) .. (210.8,25.85) .. controls (209.37,23.88) and (208.1,23.65) .. (206.97,25.16) .. controls (205.12,26.64) and (203.38,26.51) .. (201.75,24.77) .. controls (199.92,23.12) and (198.13,23.18) .. (196.4,24.95) .. controls (194.82,26.8) and (193.01,27.04) .. (190.98,25.68) .. controls (189.73,24.22) and (188.36,24.52) .. (186.87,26.57) .. controls (185.5,28.67) and (183.67,29.21) .. (181.36,28.19) .. controls (179.88,26.94) and (178.5,27.45) .. (177.23,29.72) .. controls (176.06,32.01) and (174.69,32.61) .. (173.1,31.5) .. controls (170.53,30.89) and (168.7,31.8) .. (167.63,34.24) .. controls (167.54,36.19) and (166.19,36.97) .. (163.58,36.56) .. controls (160.89,36.27) and (159.56,37.12) .. (159.57,39.1) .. controls (158.78,41.67) and (157.46,42.58) .. (155.63,41.84) .. controls (152.89,41.8) and (151.6,42.77) .. (151.77,44.76) .. controls (151.15,47.41) and (149.89,48.44) .. (147.99,47.87) .. controls (145.23,48.07) and (143.6,49.53) .. (143.11,52.26) .. controls (143.48,54.21) and (142.3,55.37) .. (139.59,55.73) .. controls (137.62,55.37) and (136.48,56.57) .. (136.19,59.33) .. controls (136.7,61.26) and (135.62,62.5) .. (132.94,63.06) .. controls (130.95,62.85) and (129.91,64.14) .. (129.84,66.91) .. controls (130.51,68.78) and (129.53,70.1) .. (126.91,70.85) .. controls (124.91,70.8) and (123.99,72.14) .. (124.16,74.89) .. controls (124.98,76.69) and (124.12,78.06) .. (121.59,79) .. controls (119.61,79.09) and (118.83,80.49) .. (119.24,83.18) .. controls (119.77,85.8) and (119.05,87.21) .. (117.1,87.41) .. controls (114.71,88.64) and (114.07,90.06) .. (115.18,91.68) .. controls (115.98,94.19) and (115.26,96.1) .. (113.01,97.43) .. controls (111.11,97.88) and (110.67,99.32) .. (111.68,101.75) .. controls (112.81,104.05) and (112.46,105.49) .. (110.62,106.07) .. controls (108.61,107.72) and (108.29,109.63) .. (109.65,111.82) .. controls (111.2,112.94) and (111.08,114.37) .. (109.27,116.1) .. controls (107.55,117.99) and (107.54,119.87) .. (109.25,121.75) .. controls (110.97,122.56) and (111.09,123.96) .. (109.61,125.93) .. controls (108.26,128.06) and (108.6,129.88) .. (110.63,131.39) .. controls (112.72,132.68) and (113.12,134.01) .. (111.81,135.39) .. controls (110.93,137.77) and (111.65,139.5) .. (113.96,140.58) .. controls (116.3,141.43) and (116.99,142.69) .. (116.02,144.34) .. controls (115.72,146.92) and (116.84,148.53) .. (119.39,149.17) .. controls (121.94,149.6) and (122.94,150.76) .. (122.4,152.63) .. controls (122.75,155.33) and (123.9,156.43) .. (125.84,155.94) .. controls (128.56,156.04) and (129.85,157.09) .. (129.71,159.09) .. controls (130.68,161.81) and (132.12,162.8) .. (134.04,162.06) .. controls (136.96,161.87) and (138.56,162.8) .. (138.84,164.85) .. controls (139.22,166.93) and (140.98,167.79) .. (144.11,167.44) .. controls (146,166.45) and (147.26,166.99) .. (147.89,169.05) .. controls (148.62,171.12) and (149.95,171.63) .. (151.9,170.57) .. controls (153.82,169.47) and (155.96,170.16) .. (158.33,172.65) .. controls (159.39,174.68) and (160.91,175.1) .. (162.9,173.91) .. controls (164.87,172.68) and (166.48,173.07) .. (167.71,175.06) .. controls (169.04,177.05) and (170.72,177.39) .. (172.76,176.09) .. controls (174.79,174.77) and (176.55,175.08) .. (178.05,177.01) .. controls (179.64,178.94) and (180.55,179.08) .. (180.78,177.43) .. controls (182.87,176.04) and (184.75,176.28) .. (186.44,178.17) -- (186.44,178.17) ;

\draw[orangewiggle]  (203.8,65.82) .. controls (205.87,67.41) and (206.21,69.09) .. (204.82,70.88) .. controls (203.36,72.62) and (203.55,74.39) .. (205.39,76.18) .. controls (207.14,77.48) and (207.21,79.06) .. (205.61,80.93) .. controls (203.96,82.79) and (203.97,84.48) .. (205.62,86.01) .. controls (207.25,87.63) and (207.2,89.21) .. (205.47,90.75) .. controls (203.72,92.28) and (203.62,93.93) .. (205.18,95.68) .. controls (206.71,97.5) and (206.56,99.19) .. (204.75,100.76) .. controls (202.92,102.31) and (202.73,104.04) .. (204.2,105.96) .. controls (205.64,107.93) and (205.42,109.69) .. (203.54,111.24) .. controls (201.7,112.34) and (201.48,113.9) .. (202.87,115.91) .. controls (204.24,117.96) and (203.99,119.52) .. (202.13,120.6) .. controls (200.18,122.1) and (199.87,123.88) .. (201.2,125.95) .. controls (202.59,127.59) and (202.29,129.14) .. (200.32,130.61) .. controls (198.33,132.04) and (197.98,133.79) .. (199.25,135.86) .. controls (200.6,137.51) and (200.27,139.02) .. (198.26,140.38) .. controls (196.25,141.69) and (195.85,143.37) .. (197.08,145.42) .. controls (198.29,147.47) and (197.89,149.09) .. (195.86,150.3) .. controls (193.84,151.45) and (193.42,153.02) .. (194.61,154.99) .. controls (195.79,156.96) and (195.31,158.62) .. (193.17,159.99) .. controls (191.15,160.91) and (190.66,162.47) .. (191.71,164.66) .. controls (192.75,166.82) and (192.21,168.4) .. (190.09,169.4) .. controls (188,170.25) and (187.41,171.78) .. (188.32,173.99) -- (186.44,178.17) ;

\draw[orangewiggle]  (203.8,65.82) .. controls (206.29,65.58) and (207.73,66.69) .. (208.14,69.16) .. controls (208.13,71.37) and (209.36,72.49) .. (211.83,72.53) .. controls (214.08,72.43) and (215.25,73.65) .. (215.33,76.18) .. controls (215.06,78.4) and (216.05,79.55) .. (218.3,79.64) .. controls (220.81,80.13) and (221.87,81.49) .. (221.48,83.74) .. controls (221.03,85.96) and (222.05,87.38) .. (224.52,88.01) .. controls (226.79,88.39) and (227.65,89.7) .. (227.12,91.93) .. controls (226.55,94.14) and (227.5,95.64) .. (229.95,96.42) .. controls (232.19,96.91) and (233,98.26) .. (232.38,100.47) .. controls (231.74,102.66) and (232.63,104.18) .. (235.05,105.03) .. controls (237.28,105.56) and (238.05,106.91) .. (237.38,109.08) .. controls (236.7,111.25) and (237.56,112.75) .. (239.95,113.58) .. controls (242.15,114.06) and (243,115.53) .. (242.51,117.98) .. controls (241.83,120.11) and (242.68,121.54) .. (245.06,122.25) .. controls (247.23,122.6) and (248.09,123.96) .. (247.62,126.34) .. controls (247.19,128.72) and (248.15,130.15) .. (250.5,130.62) .. controls (252.81,130.98) and (253.8,132.3) .. (253.45,134.58) .. controls (253.18,136.89) and (254.19,138.08) .. (256.48,138.16) .. controls (258.89,138.27) and (260.16,139.5) .. (260.27,141.87) .. controls (260.58,144.26) and (261.9,145.26) .. (264.25,144.86) .. controls (266.4,144.17) and (267.81,144.89) .. (268.48,147.02) -- (269.21,147.3) ;

\path (151.8,57.6) node [lab, anchor=north west] {$\mathcal{O}_{5}$};
\path (193.5,41.8) node [lab, anchor=north west] {$\mathcal{O}_{0}$};
\path (250.6,51.6) node [lab, anchor=north west, rotate=-9] {$\mathcal{O}_{1}$};
\path (287.4,113.4) node [lab, anchor=north west] {$\mathcal{O}_{2}$};
\path (278.1,156.9) node [lab, anchor=north west] {$\mathcal{O}_{3}$};
\path (164.0,184.4) node [lab, anchor=north west] {$\mathcal{O}_{4}$};

\path (377,97.5) node [lab, anchor=north west, font=\small] {One term in $\mathrm{ch}_{2} \ast \mathrm{ch}_{4}$};

\end{tikzpicture}
    \caption{Feynman graph integral example}
    \label{fig:ch2ch4}
\end{figure}

Suppose that $\xi_0, \ldots, \xi_d \in \lie{witt}_d$.
This graph expansion takes the following form
\begin{multline}
  \mathrm{ResTr}^{S^1}\left(\xi_0\wedge \cdots\wedge \xi_d\right) \\=
  \sum_{\permute\in S_{d+1}}\mathrm{Res}_{q} \left({\int_{S^1_{\mathrm{cyc}}[k+1]}\mathbf{Mult}\left(e^{\Pi+D}(\sigma(\xi_{\permute(0)})\cdot d\theta_0\otimes \sigma(\xi_{\permute(1)})\otimes\cdots\otimes \sigma(\xi_{\permute(d)}))\right)}|_{p=0}\right)\\
  =\sum_{\permute\in S_{d+1}}(\pm)_{\permute}\cdot\sum_{\substack{l ;\Gamma_1,\dots,\Gamma_k:\\
  |\Gamma_1|+\cdots |\Gamma_k|=d+1-l}} \frac{\int_{S^1_{\mathrm{cyc}}[d+1-l]}\omega^{\tau}_{\Gamma_1}\cdots\omega^{\tau}_{\Gamma_k} }{\#\op{Aut}(\cup^k_{i=1}\Gamma_i)}\cdot \rho\left(\tr^{\mathrm{c}_1^l}_{\Gamma_1,\dots,\Gamma_k}\left(\xi_{\permute(0)}\otimes\cdots \otimes\xi_{\permute(d)}\right)\right),
\end{multline}

We explain the notations and conventions:
\begin{itemize}
  \item[(i)] The inside sum is over integers $0 \leq l \leq d+1$ and graphs $\Gamma_1,\ldots,\Gamma_k$ that are connected, disjoint one-loop diagrams.
    Each one-loop diagram $\Gamma_i$ induces a cyclic ordering on the set of vertices denoted 
\[
  V(\Gamma_i)=\{j^i_1,\cdots,j^i_{|\Gamma_i|}\} 
\]
\item[(ii)] The differential forms which appear in the integral over $S^1_{cyc}[d+1-l]$ are defined by
\[
  \omega^{\tau}_{\Gamma_i} \define (-\pi^*_{\tau^{-1}(j^{i}_1)\tau^{-1}(j^{i}_2)}\SP)\cdots (-\pi^*_{\tau^{-1}(j^{i}_{|\Gamma_i|-1})\tau^{-1}(j^{i}_{|\Gamma_i|})}\SP)\cdot (-\pi^*_{\tau^{-1}(j^{i}_{|\Gamma_i|})\tau^{-1}(j^{i}_{1})}\SP)
\]
\item[(iii)] The trace symbol which appears inside $\rho$ is defined by
\[
\tr^{\mathrm{c}_1^l}_{\Gamma_1,\dots,\Gamma_k}\left(\xi_{\permute(0)}\otimes\cdots \otimes\xi_{\permute(d)}\right)
\]
\[
\define \tr_{\Gamma_1}\cdots \tr_{\Gamma_k}(\sigma(J\xi_{\permute(0)})\otimes\cdots \otimes (-\frac{1}{2}\mathrm{div}(\xi_{s_1}))\otimes \cdots  \otimes (-\frac{1}{2}\mathrm{div}(\xi_{s_l}))\otimes\cdots\otimes \sigma(J\xi_{\permute(d)}))
\]
where $\tau$ is a permuation such that
\[
\{\tau^{-1}(s_1),\dots,\tau^{-1}(s_l)\}=\{0,\dots,d\}-\cup^{k}_{i=1}V(\Gamma_i)
\]
and
\[
\tr_{\Gamma_1}\cdots \tr_{\Gamma_k}(M_{\permute(0)} \otimes a_{\permute(0)})\otimes\cdots \otimes (\mathrm{Id}\otimes a_{s_1})\otimes \cdots  \otimes (\mathrm{Id}\otimes a_{s_l})\otimes\ldots\otimes (M_{\permute(d)} \otimes a_{\permute(d)})
\]
\[
=\mathrm{Tr}(M_{j^1_1}\cdots M_{j^1_{|\Gamma_1|}})\cdots \mathrm{Tr}(M_{j^k_1}\cdots M_{j^k_{|\Gamma_k|}})\cdot \left(a_{\permute(0)}\otimes\cdots \otimes a_{\permute(d)}\right).
\]
\end{itemize}

Together with the Koszul rule of signs of the permuation $\tau$, the algebraic part of weight expression we have that
\[
(\pm)_{\permute}\cdot\rho\left(\tr^{\mathrm{c}_1^l}_{\Gamma_1,\dots,\Gamma_k}\left(\xi_{\permute(0)}\otimes\cdots \otimes\xi_{\permute(d)}\right)\right)=\rho\left(\tr^{\mathrm{c}_1^l}_{\Gamma_1,\dots,\Gamma_k}\left(\xi_{0}\otimes\cdots \otimes\xi_{d}\right)\right)
\]
Thus
\begin{multline}
 \sum_{\permute\in S_{d+1}}(\pm)_{\permute}\cdot\sum_{\substack{\Gamma_1,\dots,\Gamma_k:\\
  |\Gamma_1|+\cdots |\Gamma_k|=d+1-l}} \frac{\int_{S^1_{\mathrm{cyc}}[d+1-l]}\omega^{\tau}_{\Gamma_1}\cdots\omega^{\tau}_{\Gamma_k} }{\#\op{Aut}(\cup^k_{i=1}\Gamma_i)}\cdot \rho\left(\tr^{\mathrm{c}_1^l}_{\Gamma_1,\dots,\Gamma_k}\left(\xi_{\permute(0)}\otimes\cdots \otimes\xi_{\permute(d)}\right)\right)\\
 = \sum\limits_{\substack{\Gamma_1,\dots,\Gamma_k:\\
  |\Gamma_1|+\cdots |\Gamma_k|=d+1-l}} \frac{ \sum\limits_{\permute\in S_{d+1}}\int_{S^1_{\mathrm{cyc}}[d+1-l]}\omega^{\tau}_{\Gamma_1}\cdots\omega^{\tau}_{\Gamma_k} }{\#\op{Aut}(\cup^k_{i=1}\Gamma_i)}\cdot \rho\left(\tr^{\mathrm{c}_1^l}_{\Gamma_1,\dots,\Gamma_k}\left(\xi_{0}\otimes\cdots \otimes\xi_{d}\right)\right) \\
  = \sum\limits_{\substack{\Gamma_1,\dots,\Gamma_k:\\
  |\Gamma_1|+\cdots |\Gamma_k|=d+1-l}} \frac{ \int_{(S^1)^{|\Gamma_1|}}\omega_{\Gamma_1}\cdots\int_{(S^1)^{|\Gamma_k|}}\omega_{\Gamma_k} }{\#\op{Aut}(\cup^k_{i=1}\Gamma_i)}\cdot \rho\left(\tr^{\mathrm{c}_1^l}_{\Gamma_1,\dots,\Gamma_k}\left(\xi_{0}\otimes\cdots \otimes\xi_{d}\right)\right) ,
\end{multline}
Here we use the fact that the sum integral over $S^1_{\mathrm{cyc}}[d+1-l]$ over \textit{all} permutation $\permute$
simply reduces to an integral over $(S^1)^{d+1-l}$: 
\[
\sum\limits_{\permute\in S_{d+1}}\int_{S^1_{\mathrm{cyc}}[d+1-l]}\omega^{\tau}_{\Gamma_1}\cdots\omega^{\tau}_{\Gamma_k}=\int_{(S^1)^{|\Gamma_1|}}\omega_{\Gamma_1}\cdots\int_{(S^1)^{|\Gamma_k|}}\omega_{\Gamma_k}
\]
where
\[
\omega_{\Gamma_i}=\SP(\theta_{j^{i}_1}-\theta_{j^{i}_2})\cdots \SP(\theta_{j^{i}_{|\Gamma_i|-1}}-\theta_{j^{i}_{|\Gamma_i|}})\cdot \SP(\theta_{j^{i}_{|\Gamma_i|}}-\theta_{j^{i}_{1}}).
\]
\begin{multline}
  \sum\limits_{\substack{\Gamma_1,\dots,\Gamma_k:\\
  |\Gamma_1|+\cdots |\Gamma_k|=d+1-l}} \frac{ \int_{(S^1)^{|\Gamma_1|}}\omega_{\Gamma_1}\cdots\int_{(S^1)^{|\Gamma_k|}}\omega_{\Gamma_k} }{\#\op{Aut}(\cup^k_{i=1}\Gamma_i)}\cdot \rho\left(\tr^{\mathrm{c}_1^l}_{\Gamma_1,\dots,\Gamma_k}\left(\xi_{0}\otimes\cdots \otimes\xi_{d}\right)\right) \\
 = \sum_{  2\cdot i_2+\cdots+ (d+1)\cdot i_{d+1}=d+1-l}\frac{ (\int_{(S^1)^{2}}\omega_{2})^{i_2}\cdots(\int_{(S^1)^{d+1}}\omega_{d+1})^{i_{d+1}} }{\prod^{d+1}\limits_{r\geq 2}(i_r)!(2r)^{i_r}}\cdot\sum\limits_{\substack{\Gamma^r_{s}:2\leq r\leq d+1,1\leq s\leq i_r\\
 |\Gamma^r_1|=\cdots=|\Gamma^r_{i_r}|=r}}  \rho\left(\tr^{\mathrm{c}_1^l}_{\Gamma_1,\dots,\Gamma_k}\left(\xi_{0}\otimes\cdots \otimes\xi_{d}\right)\right)
\end{multline}

Recall, see for example \cite{FFS}, that
\[
I_{2r}=\int_{(S^1)^{2r}}\omega_{2r}=2\cdot (2\pi)^{-2r}(-1)^r\sum^{\infty}_{k=1}\frac{1}{k^{2r}}=2\cdot (2\pi)^{-2r}(-1)^r\cdot (-1)^{r+1}\frac{(2\pi)^{2r}}{2\cdot (2r)!}\cdot B_{2r}=-\frac{B_{2r}}{(2r)!}
\]
Thus, we finally conclude
\begin{align*}
   & \sum_{  2\cdot i_2+\cdots+ (2\lfloor\frac{d+1}{2}\rfloor)\cdot i_{2\lfloor\frac{d+1}{2}\rfloor}=d+1-l}\frac{ (-B_{2})^{i_2}\cdot (-B_{4})^{i_4}\cdots(-B_{2\lfloor\frac{d+1}{2}\rfloor})^{i_{2\lfloor\frac{d+1}{2}\rfloor}} }{l!\cdot \prod^{\lfloor\frac{d+1}{2}\rfloor}\limits_{r\geq 2}i_{2r}!\cdot (2r)^{i_{2r}}}\cdot (\frac{\mathrm{ch}_1}{2})^{l}(\mathrm{ch}_2)^{i_2}(\mathrm{ch}_4)^{i_4}\cdots (\mathrm{ch}_{2\lfloor\frac{d+1}{2}\rfloor})^{2\lfloor\frac{d+1}{2}\rfloor}\\
   & \hspace{5cm} =\op{Td}|_{2d+2}.
\end{align*}

\end{proof}

\newpage 

\printbibliography
\end{document}